\documentclass[a4paper,12pt]{amsart}

\renewcommand{\leq}{\leqslant}
\renewcommand{\le}{\leqslant}
\renewcommand{\geq}{\geqslant}
\renewcommand{\ge}{\geqslant}

\usepackage[english]{babel}
\usepackage[latin1]{inputenc}
\usepackage{amsmath}
\usepackage{amsthm}
\usepackage{amsfonts}
\usepackage{amssymb, latexsym, graphics, showidx}
\usepackage{color}

\numberwithin{equation}{section}
\newtheorem{de}{Definition}[section]
\newtheorem{thm}{Theorem}[section]

\newtheorem{prop}[thm]{Proposition}
\newtheorem{lem}[thm]{Lemma}

\newcommand{\R}{\mathbb{R}}
\newcommand{\N}{\mathbb{N}}
\newcommand{\Z}{\mathbb{Z}}

\newcommand{\lam}{\lambda}
\newcommand{\al}{\alpha}

\newcommand{\xs}{\overline{x}}
\newcommand{\ys}{\overline{y}}
\newcommand{\ts}{\overline{t}}
\newcommand{\tas}{\overline{\tau}}

\newcommand{\p}{\partial}
\newcommand{\s}{\sigma}
\newcommand{\ep}{\epsilon}
\newcommand{\I}{\mathcal{I}_s}

\newcommand{\Xs}{\overline{X}}
\newcommand{\Ys}{\overline{Y}}

\newcommand{\beq}{\begin{equation}}
\newcommand{\eeq}{\end{equation}}
\newcommand{\beqs}{\begin{equation*}}
\newcommand{\eeqs}{\end{equation*}}
\newcommand{\beqa}{\begin{eqnarray}}
\newcommand{\eeqa}{\end{eqnarray}}
\newcommand{\beqas}{\begin{eqnarray*}}
\newcommand{\eeqas}{\end{eqnarray*}}

\newcommand{\LIM}{\displaystyle\lim}
\newcommand{\SUM}{\displaystyle\sum}

\title[Homogenization and Orowan's law]{Homogenization and Orowan's law\\
for anisotropic fractional operators of any order}
\author{Stefania Patrizi and Enrico Valdinoci}
\thanks{The authors have been supported by the
ERC grant 277749 ``EPSILON Elliptic
Pde's and Symmetry of Interfaces and Layers for Odd Nonlinearities''}

\address{Weierstra{\ss} Institut f{\"u}r Angewandte und Stochastik,
Mohrenstrasse 39, D-10117 Berlin, Germany}
\email{Stefania.Patrizi@wias-berlin.de} 
\email{Enrico.Valdinoci@wias-berlin.de}

\begin{document}

\begin{abstract}
We consider an anisotropic L{\'e}vy operator $\mathcal{I}_s$
of any order $s\in(0,1)$ and we consider the homogenization
properties of an evolution equation.

The scaling properties and the effective Hamiltonian that
we obtain is different according to the cases $s<1/2$ and $s>1/2$.

In the isotropic onedimensional case, we also prove
a statement related to the so-called Orowan's law,
that is an appropriate scaling of the effective Hamiltonian
presents a linear behavior.
\end{abstract}

\subjclass[2010]{74Q15, 35B27, 35R11, 82D25.}
\keywords{Crystal dislocation, homogenization,
fractional operators.}

\maketitle

\section{Introduction}

In this paper we study an evolutionary problem run by
a fractional and possibly anisotropic operator
of elliptic type.

These type of equations arise natural in crystallography,
in which the solution of the equation has the physical
meaning of the atom dislocation inside the crystal structure,
see e.g. the detailed discussion of the Pierls-Nabarro
crystal dislocation model in~\cite{years after}.
\bigskip

Due to their mathematical interest and in view of
the concrete applications in physical models, these
problems have been extensively studied
in the recent literature, also using new methods coming
from the analysis of fractional operators,
see for instance~\cite{mp, mp2, gonzalezmonneau, dpv, dfv}
and references therein.
\bigskip

In particular, here we study an homogenization problem,
related to long-time behaviors of the system at a macroscopic scale.
The scaling of the system and the results obtained will be different
according to the fractional parameter~$s\in(0,1)$.
Namely, when~$s>1/2$
the effective Hamiltonian ``localizes'' and it only depends
on a first order differential operator.
Conversely, when~$s<1/2$, the non-local features are predominant
and the effective Hamiltonian will involve
the fractional operator of order~$s$.
That is, roughly speaking, for any~$s\in(0,1)$,
the effective Hamiltonian is
an operator of order~$\min\{2s,1\}$, which reveals the stronger
non-local effects present in the case~$s<1/2$.
\bigskip

The strong non-local features of the case~$s<1/2$
are indeed quite typical in crystal dislocation dynamics,
see~\cite{dpv} and~\cite{dfv}.
Nevertheless, for any~$s\in(0,1)$, we will be able to show
that a suitably scaled effective Hamiltonian behaves linearly
with respect to the leading operator, thus providing
an extension of the so-called Orowan's law.
\bigskip

We now recall in further detail the state of the
art for the homogenization of fractional problems
in crystal dislocation, then we introduce 
the formal setting that we deal with and present
in details our results.

In \cite{mp} Monneau and the first  author study an homogenization problem for the evolutive Pierls-Nabarro model, which is a phase field model describing dislocation dynamics. They consider the following equation 
\begin{equation}\label{uepmonneau}
\begin{cases}
\p_{t}
u^\epsilon=\mathcal{I}_1[u^\epsilon(t,\cdot)]-W'\left(\frac{u^\epsilon}{\epsilon}\right)+\s
\left(\frac{t}{\epsilon},\frac{x}{\epsilon}\right)&\text{in}\quad \R^+\times\R^N\\
u^\epsilon(0,x)=u_0(x)& \text{on}\quad \R^N,
\end{cases}
\end{equation}
where $W$ is a periodic potential and $\mathcal{I}_1$ is an anisotropic L{\'e}vy operator of order 1, which includes as  particular case the operator $-(-\Delta)^\frac{1}{2}$, and they prove that the solution 
$u^\ep$ of \eqref{uepmonneau} converges as $\ep\rightarrow 0$ to the solution $u^0$ of the following homogenized problem 
\begin{equation}\label{ueffettmonneau}
\begin{cases}
\p_{t} u=\overline{H}(\nabla_x u,\mathcal{I}_1[u(t,\cdot)])
&\text{in}\quad \R^+\times\R^N\\
u(0,x)=u_0(x)& \text{on}\quad \R^N.
\end{cases}
\end{equation}
For $\ep=1$, the solution $u^\ep$ has the physical meaning of an atom dislocation along a slip plane (the rest position of the atom lies on the lattice that is prescribed by the periodicity of the potential $W$).  The number $\ep$ describes the ratio between the microscopic scale and the macroscopic scale and then it  is a small number.  After a suitable rescaling one gets equation
\eqref{uepmonneau}. The limit  $u^0$ can be interpreted as a macroscopic plastic strain satisfying the macroscopic plastic flow rule \eqref{ueffettmonneau}. The function 
$\overline{H}$, usually called effective Hamiltonian, is determined, as usual in homogenization, by a cell problem, which is in this case, for  $p\in\R^N$ 
and $L\in\R$, the following:
\begin{equation}\label{vmaunneau}
\begin{cases}
\lam+\p_{\tau} v=\mathcal{I}_1[v(\tau,\cdot)]+L-W'(v+\lam \tau+p\cdot y)+\s
(\tau,y)&\text{in}\quad \R^+\times\R^N\\ v(0,y)=0& \text{on}\quad
\R^N.
\end{cases}
\end{equation}
For  any $p\in\R^N$ 
and $L\in\R$, the quantity $\lambda=\lambda(p,L)$ is the unique number for which there exists a solution $v$ of \eqref{vmaunneau} which is bounded in $\R^+\times\R^N$.
Therefore,
the function  $\overline{H}(p,L):=\lambda(p,L)$
is well defined, and, in addition, this function
turns out to be continuous and non-decreasing in $L$.

In a second paper \cite{mp2}, the authors consider, as a
particular case, the one in which~$N=1$,
$\mathcal{I}_1=-(-\Delta)^\frac{1}{2}$ is the half Laplacian and $\sigma\equiv 0$, and they study the behavior of $\overline{H}(p,L)$ for small $p$ and $L$. In this regime they recover the Orowan's law, which claims that 
\beqs \overline{H}(p,L)\sim c_0|p|L\eeqs
for some constant of proportionality $c_0>0$.  To show this last result, estimates on the layer solution associated to $-(-\Delta)^\frac{1}{2}$, i.e. on the solution $\phi$ of 
\begin{equation}\label{phimonneau}
\begin{cases}-(-\Delta)^\frac{1}{2}\phi=W'(\phi)&\text{in}\quad \R\\
\phi'>0&\text{in}\quad \R\\
\LIM_{x\rightarrow-\infty}\phi(x)=0,\quad\LIM_{x\rightarrow+\infty}\phi(x)=1,\quad\phi(0)=\frac{1}{2},
\end{cases}
\end{equation}
are needed. Such estimates were  proved in \cite{gonzalezmonneau} under suitable assumptions on $W$,  while the existence of a unique solution $\phi$ 
of \eqref{phimonneau} was proved in \cite{csm}.

Recently, these kind of  estimates have been proved for layer solutions associated to the fractional Laplacian  $-(-\Delta)^s$  for  $s\in(0,1)$ 
by Palatucci, Savin and the second author in \cite{psv}.  More general
results on $\phi$ were obtained 
by Dipierro, Palatucci and the second author in~\cite{dpv} for
the case $s\in\left[\frac{1}{2},1\right)$. See also~\cite{cs} for
related results.

In this paper, in view of these new estimates, we want to extend the results of \cite{mp} and \cite{mp2} to the case where the non-local operator in \eqref{uepmonneau} is an 
 anisotropic L{\'e}vy operator of any order $s\in(0,1)$.  Precisely, given $\varphi\in C^2(\R^N)\cap L^\infty(\R^N)$, let us define 
\beq\label{slapla} \I[\varphi](x):=PV\int_{\R^N}\frac{\varphi(x+y)-\varphi(x)}{|y|^{N+2s}}g\left(\frac{y}{|y|}\right)dy,\eeq
where $PV$ stands for the principal value of the integral and the function $g$ satisfies
\begin{itemize}
\item [(H1)] $g\in C({\bf S}^{N-1}),\,g> 0$,\,$g$ even.\end{itemize} 
When $g\equiv C(N,s)$ with $ C(N,s)$ suitable constant depending on the dimension N and on the exponent  $s$, then \eqref{slapla} is the integral representation of 
$-(-\Delta)^s$. 
 
In addition to (H1) we make the following assumptions:
\begin{itemize}
\item [(H2)]$W\in C^{1,1}(\R)$ and $W(v+1)=W(v)$ for any $v\in\R$;
\item [(H3)]$\sigma\in C^{0,1}(\R^+\times\R^N)$ and $\s(t+1,x)=\s(t,x)$, $\sigma(t,x+k)=\sigma(t,x)$ for any
$k\in\Z^N$ and $(t,x)\in\R^+\times\R^N$;
\item [(H4)]$u_0\in W^{2,\infty}(\R^N)$.
\end{itemize}
For $s>\frac{1}{2}$ we consider the following homogenization problem:
\begin{equation}\label{ueps>1/2}
\begin{cases}
\p_{t}
u^\epsilon=\ep^{2s-1}\I[u^\epsilon(t,\cdot)]-W'\left(\frac{u^\epsilon}{\epsilon}\right)+\s
\left(\frac{t}{\epsilon},\frac{x}{\epsilon}\right)&\text{in}\quad \R^+\times\R^N\\
u^\epsilon(0,x)=u_0(x)& \text{on}\quad \R^N,
\end{cases}
\end{equation}
and for $s<\frac{1}{2}$:

\begin{equation}\label{ueps<1/2}
\begin{cases}
\p_{t}
u^\epsilon=\I[u^\epsilon(t,\cdot)]-W'\left(\frac{u^\epsilon}{\epsilon^{2s}}\right)+\s
\left(\frac{t}{\epsilon^{2s}},\frac{x}{\epsilon}\right)&\text{in}\quad \R^+\times\R^N\\
u^\epsilon(0,x)=u_0(x)& \text{on}\quad \R^N.
\end{cases}
\end{equation}
Remark that the scalings for $s>\frac{1}{2}$ and  $s<\frac{1}{2}$ are different.  They formally coincide when $s=\frac{1}{2}$. 
We prove that the solution $u^\ep$ of \eqref{ueps>1/2} converges as $\ep\rightarrow 0$ to the solution $u^0$ of the homogenized problem 
\begin{equation}\label{ueffetts>1/2}
\begin{cases}
\p_{t} u=\overline{H}_1(\nabla_x u)
&\text{in}\quad \R^+\times\R^N\\
u(0,x)=u_0(x)& \text{on}\quad \R^N,
\end{cases}
\end{equation}
with an effective Hamiltonian $\overline{H}_1$ which does not depend on $\I$ anymore, while the solution 
$u^\ep$ of \eqref{ueps<1/2} converges as $\ep\rightarrow 0$ to  $u^0$ solution of the following
\begin{equation}\label{ueffetts<1/2}
\begin{cases}
\p_{t} u=\overline{H}_2(\I [u])
&\text{in}\quad \R^+\times\R^N\\
u(0,x)=u_0(x)& \text{on}\quad \R^N,
\end{cases}\end{equation}
 with an effective Hamiltonian $\overline{H}_2$  not depending  on the gradient. 
As we will see, the functions  $\overline{H}_1$ and $\overline{H}_2$ are determined by the following cell problem:
\begin{equation}\label{v}
\begin{cases}
\lam+\p_{\tau} v=\I[v(\tau,\cdot)]+L-W'(v+\lam \tau+p\cdot y)+\s
(\tau,y)&\text{in}\quad \R^+\times\R^N\\ v(0,y)=0& \text{on}\quad
\R^N,
\end{cases}
\end{equation} that is $\overline{H}_1$ and $\overline{H}_2$ are determined by the unique $\lam$ for which \eqref{v} possesses a bounded solution (according to the cases $s>\frac{1}{2}$ and 
$s<\frac{1}{2}$, respectively). We observe that the solutions of \eqref{ueffetts>1/2} and \eqref{ueffetts<1/2} may have quite different behaviors, since $\nabla u$ and $\I[u]$ may be very different at a given point, even in dimension 1 and when $s$ is close to $\frac{1}{2}$ (see for instance \cite{dsv}).
Following  \cite{mp}, in order to solve \eqref{v}, we show for any $p\in\R^N$ and $L\in\R$ the existence of a unique solution of 

\begin{equation}\label{w}
\begin{cases}
\p_{\tau} w=\I[w(\tau,\cdot)]+L-W'(w+p\cdot y)+\s
(\tau,y)&\text{in}\quad \R^+\times\R^N\\ w(0,y)=0& \text{on}\quad
\R^N,
\end{cases}
\end{equation}
and we look for some $\lam$ such that $w-\lam\tau$ is bounded. Precisely we have:
\begin{thm}[Ergodicity]\label{ergodic}Assume (H1)-(H4). For $L\in\R$ and $p\in\R^N$, there
exists a unique viscosity solution $w\in C_b(\R^+\times\R^N)$ of
\eqref{w} and there exists a unique $\lam\in\R$ such that $w$
satisfies: 
\begin{equation*}
{\mbox{$\displaystyle\frac{w(\tau,y)}{\tau}$ converges towards $\lam$ as
$\tau\rightarrow+\infty$, locally uniformly in $y$.}}\end{equation*} The real
number  $\lam$ is denoted by $\overline{H}(p,L)$. The function
$\overline{H}(p,L)$ is continuous on $\R^N\times\R$ and
non-decreasing in $L$.
\end{thm}
Once  the cell problem is solved, we can prove the following convergence results:

\begin{thm}[Convergence for $s>\frac{1}{2}$]\label{convergence}Assume (H1)-(H4). The solution
$u^\ep$ of \eqref{ueps>1/2} converges towards the solution $u^0$ of
\eqref{ueffetts>1/2} locally uniformly in $(t,x)$, where $$\overline{H}_1(p):=\overline{H}(p,0)$$ and $\overline{H}(p,L)$
is defined in Theorem \ref{ergodic}.
\end{thm}

\begin{thm}[Convergence for $s<\frac{1}{2}$]\label{convergences<1/2}Assume (H1)-(H4). The solution
$u^\ep$ of \eqref{ueps<1/2} converges towards the solution $u^0$ of
\eqref{ueffetts<1/2} locally uniformly in $(t,x)$, where $$\overline{H}_2(L):=\overline{H}(0,L)$$ and $\overline{H}(p,L)$
is defined in Theorem \ref{ergodic}.
\end{thm}

We point out that the effective Hamiltonians $\overline{H}_1$ and $\overline{H}_2$ represent the speed of propagation of the dislocation dynamics according to 
\eqref{ueffetts>1/2} and \eqref{ueffetts<1/2}. In particular, due to Theorems \ref{convergence} and \ref{convergences<1/2}, such speed only depends on the slope of the dislocation 
in the weakly non-local setting   $s>\frac{1}{2}$ and only on an operator of order $s$ of the dislocation in the strongly non-local setting $s<\frac{1}{2}$.

We will next  consider the case: $N=1$,  $\I=-(-\Delta)^s$  and $\sigma\equiv 0$, and we will make the further following assumptions on the potential $W$:
\begin{equation}\label{W}
\begin{cases}W\in C^{4,\beta}(\R)& \text{for some }0<\beta<1\\
W(v+1)=W(v)& \text{for any } v\in\R\\
W=0& \text{on }\Z\\
W>0 & \text{on }\R\setminus\Z\\
\al=W''(0)>0\\
W\text{ is even  if }s\in\left(0,\frac{1}{2}\right).\\
\end{cases}
\end{equation}
Under assumption (\ref{W}), it is  known, see  \cite{cs} and \cite{psv},
that there exists a unique function $\phi$ solution of
\begin{equation}\label{phi}
\begin{cases}\I[\phi]=W'(\phi)&\text{in}\quad \R\\
\phi'>0&\text{in}\quad \R\\
\LIM_{x\rightarrow-\infty}\phi(x)=0,\quad\LIM_{x\rightarrow+\infty}\phi(x)=1,\quad\phi(0)=\frac{1}{2}.
\end{cases}
\end{equation}
Then we can prove the following extension of the Orowan's law:
\begin{thm}\label{hullprop}Assume \eqref{W} 
and let $p_0,\,L_0\in\R$ with $p_0\neq 0$. Then the function $\overline{H}$ defined in Theorem \ref{ergodic} satisfies
\begin{equation}\label{orowan}\frac{\overline{H}(\delta p_0,\delta^{2s} L_0)}{\delta^{1+2s}}\rightarrow
c_0|p_0|L_0\quad\text{as }\delta\rightarrow0^+ \quad \mbox{with}\quad c_0=\left(\int_\R (\phi')^2\right)^{-1}.\end{equation}
\end{thm}
We notice that \eqref{orowan} can be rephrased using the notation $p:=\delta p_0$ and $L:=\delta L_0$, by saying
\begin{equation*}{\mbox{$\displaystyle \overline{H}(p,L)=c_0|p| L$ + higher order terms, }}\end{equation*}
which in particular shows that $\overline{H}$ has a linear growth close to the origin.
We observe that assumption (\ref{W}) is stronger than (H2), since it requires the minima to be non-degenerate, it assumes further smoothness on the potential 
and the even property in the case $s<\frac{1}{2}$. This last property is natural for physical applications, since typically the effect of a dislocation 
in a given direction compensates with the one in the opposite direction (in particular it is satisfied in the classical Peierls-Nabarro model in which $W(u)=1-\cos(2\pi u)$).
From the technical point of view, this property is needed only in the strongly non-local case $s<\frac{1}{2}$ since the first order asymptotic decay of the layer solution \eqref{phi}
lies below a critical threshold (the even property allows us to deduce a useful second order approximation).

The rest of the paper is organized as follows.
First we recall some definitions and basic fact about viscosity
solutions. Then, in Section~\ref{ansatzsec}
we imbed our problem into one in one dimension more,
to keep track of all the homogenized quantities,
and we state the ansatz on the solution we look for.
The corrector equation will be studied in Section~\ref{CO},
where Theorem \ref{ergodic} will be proved.
Thus, we will prove
Theorems~\ref{convergence} and~\ref{convergences<1/2}
in Sections~\ref{convprofs>1/2sec}
and~\ref{yuio9}, respectively. Then we present the extension of
the Orowan's law and the
proof of Theorem \ref{hullprop} in Section~\ref{OR}.

\subsection{Notations and definition of viscosity solution} We denote by $B_r(x)$ the ball of radius
$r$ centered at $x$. The cylinder $(t-\tau,t+\tau)\times B_r(x)$
is denoted by $Q_{\tau,r}(t,x)$.

$\lfloor x \rfloor$ and $\lceil x\rceil$ denote respectively the
floor and the ceiling
integer part functions of a real number $x$.

It is convenient to introduce the singular measure defined on
$\R^N\setminus\{0\}$ by
$$\mu(dz)=\frac{1}{|z|^{N+2s}}g\left(\frac{z}{|z|}\right)dz,$$and
to denote
$$\I^{1,r}[\varphi,x]=\int_{|z|\leq r}(\varphi(x+z)-\varphi(x)-\nabla \varphi(x)\cdot
z)\mu(dz),$$
$$\I^{2,r}[\varphi,x]=\int_{|z|>r}(\varphi(x+z)-\varphi(x))\mu(dz).$$

For a function $u$ defined on $(0,T)\times\R^N$, $0<T\leq+\infty$,
for $0<\al<1$ we denote by $<u>_x^\al$  the seminorm defined by
$$<u>_x^\al:=\sup_{(t,x''),\,(t,x')\in(0,T)\times\R^N\atop x''\neq x'}\frac{|u(t,x'')-u(t,x')|}{|x''-x'|^\al}$$
and by $C_x^{\al}((0,T)\times\R^N)$ the space of continuous
functions defined on $(0,T)\times\R^N$ that are bounded and with
bounded seminorm
 $<u>_x^\al$.

Finally, we denote by $USC_b(\R^+\times\R^N)$ (resp.,
$LSC_b(\R^+\times\R^N)$) the set of upper (resp., lower)
semicontinuous functions on $\R^+\times\R^N$ which are bounded on
$(0,T)\times\R^N$ for any $T>0$ and we set
$C_b(\R^+\times\R^N):=USC_b(\R^+\times\R^N)\cap
LSC_b(\R^+\times\R^N)$.

Let us conclude by recalling  the definition of viscosity solution for a general
first order non-local equation with associated initial
condition:
\begin{equation}\label{generalpb}
\begin{cases}
u_t=F(t,x,u,Du,\I[u])&\text{in}\quad \R^+\times\R^N\\
u(0,x)=u_0(x)& \text{on}\quad \R^N,
\end{cases}
\end{equation}where $F(t,x,u,p,L)$ is continuous and
non-decreasing in $L$. The definition relies on the following observation:
if $\varphi$ is a bounded $C^2$ function, then for any $r>0$
\beqs\begin{split}\I[\varphi,x]&=\int_{|z|\leq r}(\varphi(x+z)-\varphi(x)-\nabla \varphi(x)\cdot
z)\mu(dz)+\int_{|z|>r}(\varphi(x+z)-\varphi(x))\mu(dz)
\\&=\I^{1,r}[\varphi,x]+\I^{2,r}[\varphi,x]\end{split}\eeqs and this expression is independent of $r$ because of the antisymmetry 
of $\nabla \varphi(x)\cdot
z\mu(dz)$. 
\begin{de}[viscosity solution]\label{defviscosity}A function $u\in USC_b(\R^+\times\R^N)$ (resp., $u\in LSC_b(\R^+\times\R^N)$) is a
viscosity subsolution (resp., supersolution) of
\eqref{generalpb} if $u(0,x)\leq (u_0)^*(x)$ (resp., $u(0,x)\geq
(u_0)_*(x)$) and for any $(t_0,x_0)\in\R^+\times\R^N$, any
$\tau\in(0,t_0)$ and any test function $\varphi\in
C^2(\R^+\times\R^N)$ such that $u-\varphi$ attains a local maximum
(resp., minimum) at the point $(t_0,x_0)$ on
$Q_{(\tau,r)}(t_0,x_0)$, then we have
\beqs\begin{split}&\p_t\varphi(t_0,x_0)-F(t_0,x_0,u(t_0,x_0),\nabla_x\varphi(t_0,x_0),\I^{1,r}[\varphi(t_0,\cdot),x_0]+\I^{2,r}[u(t_0,\cdot),x_0])\leq
0\\&\text{(resp., }\geq 0),\end{split}\eeqs for a positive number $r$.  A function $u\in
C_b(\R^+\times\R^N)$ is a viscosity solution of
\eqref{generalpb} if it is a viscosity sub and supersolution
of \eqref{generalpb}.
\end{de}
One can prove that Definition \ref{defviscosity} does not depend on $r$ and if the  inequality above is satisfied for a given $r>0$, then it is satisfied for any $r>0$, see \cite{mp} and the references therein.


\section{The Ansatz}\label{ansatzsec}
As explained in \cite{mp}, because of the presence of the term $W'\left(\frac{u^\epsilon}{\epsilon}\right)$ in \eqref{ueps>1/2} and  \eqref{ueps<1/2}, in order to get the homogenization results, we need to imbed our problems into higher dimensional ones. Let us first assume $s>\frac{1}{2}$. Then  we will consider: 
\begin{equation}\label{Ueps>1/2}
\begin{cases}
\p_{t}
U^\epsilon=\ep^{2s-1}\I[U^\epsilon(t,\cdot,x_{N+1})]-W'\left(\frac{U^\epsilon}{\epsilon}\right)+\s
\left(\frac{t}{\epsilon},\frac{x}{\epsilon}\right)&\text{in}\quad \R^+\times\R^{N+1}\\
U^\epsilon(0,x,x_{N+1})=u_0(x)+x_{N+1}& \text{on}\quad \R^{N+1}
\end{cases}
\end{equation}
and we will prove that $U^\ep$ converges as $\ep\rightarrow 0$ to the function
$$U^0 (t,x,x_{N+1})=u^0(t,x)+ x_{N+1}$$  with $u^0$ the solution of \eqref{ueffetts>1/2}. We remark that $U^0 $
 satisfies:
  \begin{equation}\label{Ueffetts>1/2}
\begin{cases}
\p_{t} U=\overline{H}_1(\nabla_x U)
&\text{in}\quad \R^+\times\R^{N+1}\\
U(0,x,x_{N+1})=u_0(x)+x_{N+1}& \text{on}\quad \R^{N+1}.
\end{cases}
\end{equation}
 The convergence of $U^\ep$ to $U^0$ will imply the converge of $u^\ep$ to $u^0$. 
 In order to prove this result, 
we  introduce the higher dimensional cell problem: for $P=(p,1)\in\R^{N+1}$ and $L\in\R$:
 \begin{equation}\label{V}\left\{
  \begin{array}{ll}
    \lam+\p_{\tau} V=L+\I[V(\tau,\cdot,y_{N+1})]-W'(V+P\cdot Y+\lam\tau)+\s(\tau,y) & \hbox{in } \R^+\times\R^{N+1}\\
    V(0,Y)=0 & \hbox{on }\R^{N+1}.
  \end{array}
\right.\end{equation} 
Here we use the notation $Y=(y,y_{N+1})$.
The right Ansatz for $U^\ep$ solution of \eqref{Ueps>1/2}, turns out to be
 \beq\label{tildeUep} U^\ep(t,x,x_{N+1})\simeq \tilde{U}^\ep(t,x,x_{N+1}):=U^0(t,x,x_{N+1})+\ep
V\left(\frac{t}{\ep},\frac{x}{\ep},\frac{U^0(t,x,x_{N+1})-\lam
t-p\cdot x}{\ep  }\right)\eeq
with $V$ the bounded solution of   \eqref{V},  for suitable values of $p$ and $L$.   
Let us verify it. 

Fix $P_0=(t_0,x_0,x_{N+1}^0)\in\R^+\times\R^{N+1}$ and let $\tilde{U}^\ep(t,x,x_{N+1})$ be defined as in \eqref{tildeUep}.
Let us denote 
 \beq\label{lamplans}\lam=\p_t U^0(P_0),\quad p=\nabla_x U^0(P_0),\eeq
 and
  $$F(t,x,x_{N+1})=U^0(t,x,x_{N+1})-\lam
t-p\cdot x,\quad  \tau=\frac{t}{\ep},\,y=\frac{x}{\ep},\quad y_{N+1}=\frac{F(t,x,x_{N+1})}{\ep }.$$
We remark that $P=(p,1)=\nabla_{(x,x_{N+1}) }U^0(P_0)$ and 
$$\frac{\tilde{U}^\ep(t,x,x_{N+1})}{\ep}=V(\tau,y,y_{N+1})+\lam\tau+p\cdot y+ y_{N+1}=V(\tau,Y)+P\cdot Y+\lam\tau.$$
Here we assume for simplicity that $U^0$ and $V$ are smooth. 
The proof of convergence consists in showing that  $\tilde{U}^\ep$ is a solution of \eqref{Ueps>1/2} in a cylinder  $(t_0-r,t_0+r)\times B_r(x_0,x_{N+1}^0)$ for $r>0$ small enough, up to an error that goes to 0 as $r\rightarrow 0^+$.  This will allow us to compare $U^\epsilon$ with $\tilde{U}^\ep$ and, thanks to the boundedness of $V$, to conclude that $U^\epsilon$ converges to $U^0$ as $\ep\rightarrow 0$.  

When we plug $\tilde{U}^\ep$ into \eqref{Ueps>1/2},  we find the equation 
\beqs \begin{split} \lam + \p_{\tau} V(\tau,Y)&=\ep^{2s-1}\I[U^0(t,\cdot,x_{N+1}),x]+\I[V(\tau,\cdot,y_{N+1}),y]\\&-W'(V+PY+\lam\tau)+\s(\tau,y)+\theta_r,
\end{split}\eeqs
where 
\beqs   \begin{split} \theta_r &= (\p_tU^0(P_0)-\p_t U^0(t,x,x_{N+1}))(\p_{y_{N+1}}V(\tau,Y)+1)
\\&+\ep^{2s}\I\left[V\left(\frac{t}{\ep},\frac{\cdot}{\ep}, \frac{F(t,\cdot, x_{N+1})}{\ep}\right),x\right]-\I[V(\tau,\cdot,y_{N+1}),y].\end{split} \eeqs
If $V$ is solution of \eqref{V} with $p$  as in \eqref{lamplans} and $L=0$, and $U^0$ satisfies $\p_t U^0(P_0)=\lam=\overline{H}(\nabla_x U^0(P_0),0)$, 
then     $\tilde{U}^\ep$ will be a solution of  \eqref{Ueps>1/2} up to  small errors $\ep^{2s-1}\I[U^0(t,\cdot,x_{N+1}),x]=o_\ep(1)$ as $\ep\rightarrow 0$ and  $\theta_r=o_r(1)$ as $r\rightarrow0^+$.  As we will see in Section \ref{convprofs>1/2sec}, 
 this last property holds true if the corrector $V$  satisfies: 
$|V|$, $|\p_{y_{N+1}}V|\leq C$ in $\R^+\times\R^{N+1}$ for some $C>0$, and 
\begin{equation*}
\p_{y_{N+1}}V(\tau,\cdot,\cdot) \quad \mbox{is H\"{o}lder continuous, uniformly in time.}
\end{equation*}
Since in \eqref{V} the quantity $\I[V(\tau,\cdot,y_{N+1})]$ is computed only in the $y$ variable, we cannot expect this kind of regularity
 for the correctors.
       Nevertheless, following \cite{mp}, we are able to construct regular
 approximated sub and supercorrectors, i.e., sub and supersolutions of approximate $N+1$-dimensional cell problems, and this is enough to conclude. 
 
 Similarly for   $s<\frac{1}{2}$, we will consider:
 \begin{equation}\label{Ueps<1/2}
\begin{cases}
\p_{t}
U^\epsilon=\I[U^\epsilon(t,\cdot,x_{N+1})]-W'\left(\frac{U^\epsilon}{\epsilon^{2s}}\right)+\s
\left(\frac{t}{\epsilon^{2s}},\frac{x}{\epsilon}\right)&\text{in}\quad \R^+\times\R^{N+1}\\
U^\epsilon(0,x,x_{N+1})=u_0(x)+x_{N+1}& \text{on}\quad \R^{N+1},
\end{cases}
\end{equation}
 and we will show that $U^\ep$ converges as $\ep\rightarrow 0$ to the function
$$U^0 (t,x,x_{N+1})=u^0(t,x)+ x_{N+1}$$  with $u^0$ the solution of \eqref{ueffetts<1/2}. Here $U^0$ is solution of 
\begin{equation}\label{Ueffetts<1/2}
\begin{cases}
\p_{t} U=\overline{H}_2(\I[U(t,\cdot,x_{N+1})])
&\text{in}\quad \R^+\times\R^{N+1}\\
U(0,x,x_{N+1})=u_0(x)+x_{N+1}& \text{on}\quad \R^{N+1}.
\end{cases}
\end{equation}
In this case,  the right Ansatz  turns out to be
\beqs U^\ep(t,x,x_{N+1})\simeq U^0(t,x,x_{N+1})+\ep^{2s}
V\left(\frac{t}{\ep^{2s} },\frac{x}{\ep},\frac{U^0(t,x,x_{N+1})-\lam t}{\ep^{2s}  }\right)\eeqs 
where  $V$ is   the bounded solution of  \eqref{V} for $p=0$ and $L=\I[U^0(t,\cdot,x_{N+1}),x].$
 
\section{Correctors}\label{CO}

In this section we prove Theorem \ref{ergodic} and the existence of smooth approximated sub and supersolutions of the higher dimensional cell problem \eqref{V} introduced in Section \ref{ansatzsec} which are needed to show the convergence Theorems \ref{convergence} and  \ref{convergences<1/2}.
The proof of these results is  given in  \cite{mp} for the case $s=1$ and it is essentially based on the comparison principle and invariance under integer translations. 
Therefore it can be easily extended to the case $s\in(0,1)$ and  for this reason, here we only give a sketch of it. \\

\noindent{\bf Step 1: Lipschitz correctors.}

One introduces the problem: for $\eta\geq0$, $a_0,\,L\in\R$,
$p\in\R^N$ and $P=(p,1)$
\begin{equation}\label{wlip}\left\{
  \begin{array}{ll}
    \p_{\tau} U=L+\I[U(\tau,\cdot,y_{N+1})]-W'(U+P\cdot Y)+\s(\tau,y)\\
    \quad\quad\,+\eta[a_0+\inf_{Y'}U(\tau,Y')-U(\tau,Y)]|\p_{y_{N+1}}U+1|
    & \hbox{in } \R^+\times\R^{N+1}\\
    U(0,Y)=0 & \hbox{on }\R^{N+1},
  \end{array}
\right.\end{equation}
and show the existence of the viscosity solution $U_\eta\in C_b(\R^+\times\R^{N+1})$. When $\eta>0$ this solution turns out to be Lipschitz continuous
 in the variable $y_{N+1}$ with 
 \beqs-1\leq \p_{y_{N+1}}U_\eta(\tau,Y)\leq
\frac{\|W''\|_\infty}{\eta}.\eeqs\\
See the proof of Propositions 6.2, 6.3 and 6.4 in \cite{mp} for  details about the existence and regularity of the solution of \eqref{wlip}.
As we will explain in Step 5, choosing conveniently  the number $a_0$ in   \eqref{wlip}, we obtain sub and supersolutions of the $N+1$-dimensional cell problem \eqref{V} which are Lipschitz continuous in $y_{N+1}$.\\

\noindent{\bf Step 2: Ergodicity.}

Using the comparison principle, and the periodicity of $\sigma$ and $W$, one can prove the following ergodic result:
\begin{prop}[Ergodic properties]\label{ergodic2}
There exists a unique $\lam_\eta=\lam_\eta(p,L)$ such that the
viscosity solution $U_\eta\in C_b(\R^+\times\R^{N+1})$  of
\eqref{wlip}  with $\eta\geq 0$, satisfies:
\begin{equation}\label{w-lam}|U_\eta(\tau,Y)-\lam_\eta\tau|\leq
C\text{ for all }\tau>0,\,Y\in \R^{N+1},\end{equation}with $C$
independent of $\eta$.
Moreover
\begin{equation}\label{lambounds}L-\|W'\|_\infty-\|\s\|_\infty+\eta a_0\leq\lam_\eta\leq
L+\|W'\|_\infty+\|\s\|_\infty+\eta a_0.\end{equation}
\end{prop}
Proposition \ref{ergodic2} can be proved like Proposition 6.4 in \cite{mp}.\\

\noindent{\bf Step 3: Proof of Theorem \ref{ergodic}.}

Let $U$ be the solution of \eqref{wlip}  with $\eta=0$, then 
the function $$w(\tau,y):=U(\tau,y,0)$$ is the solution of \eqref{w} and by Proposition \ref{ergodic2}, there exists  a unique  $\lam$ such that 
\beq\label{wergodic}|w(\tau,y)-\lambda\tau|\leq C.\eeq This property implies that $\lam$ is the unique number such that 
$w(\tau,y)/\tau$ converges towards $\lambda$ as $\tau\rightarrow +\infty$, and Theorem \ref{ergodic} is proved.

The next two steps are only needed in the proof of Theorems \ref{convergence} and  \ref{convergences<1/2}. We first state some properties of the effective Hamiltonian, then 
in Step 5, we construct approximate sub and supersolutions of \eqref{V} which are smooth also in the additional variable $y_{N+1}$. This further
 regularity property is needed to control the error when we compare the solution $U^\ep$ of \eqref{Ueps>1/2} and \eqref{Ueps<1/2} with the corresponding ansatz, as explained 
 in Section \ref{ansatzsec}.\\

\noindent{\bf Step 4: Properties of the effective Hamiltonian}

We have 
\begin{prop}[Properties of the effective Hamiltonian]\label{Hprop} Let $p\in\R^N$ and $L\in\R$. Let $\overline{H}(p,L)$ be the constant
defined by Theorem \ref{ergodic}, then
$\overline{H}:\R^N\times\R\rightarrow\R$ is a continuous function
with the following properties:
\begin{itemize}
\item[(i)]$\overline{H}(p,L)\rightarrow {\pm}\infty$ as $L \rightarrow
{\pm}\infty$ for any $p\in\R^N$;
\item[(ii)] $\overline{H}(p,\cdot)$ is non-decreasing on $\R$ for any $p\in\R^N$;
\item[(iii)]If $\s(\tau,y)=\s(\tau,-y)$ then
$$\overline{H}(p,L)=\overline{H}(-p,L);$$
\item[(iv)]If $W'(-s)=-W'(s)$ and $\s(\tau,-y)=-\s(\tau,y)$ then
$$\overline{H}(p,-L)=-\overline{H}(p,L).$$
\end{itemize}
\end{prop}

For the proof of Proposition \ref{Hprop} see Proposition 5.4 in \cite{mp}.\\

\noindent{\bf Step 5: Construction of smooth approximate sub and supercorrectors.}

The ergodic property \eqref{wlip} of $U_\eta$ implies that there exists $C_1>0$ such that 
$$C_1+\inf_{Y'}U_\eta(\tau,Y')-U_\eta(\tau,Y)>0,$$ for any $\eta>0$.
Then, one take  $U^+_\eta$ to be  the solution of \eqref{wlip} with $a_0=C_1$ and 
$U^-_\eta$ to be  the solution of \eqref{wlip} with $a_0=0$. We remark that $U^+_\eta$ and $U^-_\eta$ are respectively super and subsolution of 
$$  \p_{\tau} U=L+\I[U(\tau,\cdot,y_{N+1})]-W'(U+P\cdot Y)+\s(\tau,y).$$
Let 
$\lam_\eta^+=\LIM_{\tau\rightarrow+\infty}\frac{U^+_\eta(\tau,Y)}{\tau}$
and
$\lam_\eta^-=\LIM_{\tau\rightarrow+\infty}\frac{U^-_\eta(\tau,Y)}{\tau}$, whose 
the existence  is guaranteed by
Proposition \ref{ergodic2}.  Stability results and the ergodic property \eqref{w-lam} imply that $\lam^+_\eta,\lam^-_\eta\rightarrow\lam$ as $\eta\rightarrow0$, with $\lam$ given by Theorem \ref{ergodic}. 

Next, one  set 
$$W^+_\eta(\tau,Y):=U^+_\eta(\tau,Y)-\lam^+_\eta\tau$$
and
$$W^-_\eta(\tau,Y):=U^-_\eta(\tau,Y)-\lam^-_\eta\tau.$$
Then $W^+_\eta$ and $W^-_\eta$ are respectively super and subsolution of \eqref{V} with respectively $\lam=\lam^+_\eta$ and $\lam=\lam^-_\eta$, and are Lipschitz continuous in the variable $y_{N+1}$. One can in addition show that 
these functions are of class $C^{\al}$ with respect to $y$ uniformly  in $y_{N+1}$, for $0<\al<\min\{1,2s\}$. This comes from Proposition 4.7 in \cite{mp} that can be easily adapted to the case $s\in(0,1)$. 

The regularity properties of $W^+_\eta$ and $W^-_\eta$ are not enough in order to prove the convergence results, Theorems \ref{convergence} and \ref{convergences<1/2}, as pointed out in Section \ref{ansatzsec}.  Therefore, one introduces a  positive smooth function
$\rho:\R\rightarrow\R$, with support in $B_1(0)$ and mass 1 and 
defines a sequence of mollifiers $(\rho_\delta)_\delta$ by
$\rho_\delta(r)=\frac{1}{\delta}\rho\left(\frac{r}{\delta}\right)$,
$r\in\R.$ 
Then, one finally defines
\beqs
V^{\pm}_{\eta,\delta}(t,y,y_{N+1}):=W^{\pm}_\eta(t,y,\cdot)\star\rho_\delta(\cdot)
=\int_{\R}W^{\pm}_\eta(t,y,z)\rho_\delta(y_{N+1}-z)dz.\eeqs
Choosing properly $\delta=\delta(\eta)$, one can prove the following result:
\begin{prop}[Smooth approximate correctors]\label{apprcorrectors}
Let $\lam$ be the constant defined by Theorem \ref{ergodic}. For
any fixed $p\in\R^N$, $P=(p,1)$, $L\in\R$ and $\eta>0$ small
enough, there exist real numbers $\lam^+_\eta(p,L)$,
$\lam^-_\eta(p,L)$, a constant $C>0$ (independent of $\eta,\,p$
and $L$) and bounded super and subcorrectors $V^+_{\eta},
V^-_{\eta}$, i.e. respectively a super and a subsolution of
\begin{equation}\label{apprcorrequ}\left\{
  \begin{array}{ll}
    \lam^{\pm}_\eta+\p_{\tau}
    V^{\pm}_\eta=L+\I[V^{\pm}_\eta(\tau,\cdot,y_{N+1})]\\
    \qquad\qquad\qquad-W'(V^{\pm}_\eta+P\cdot Y+\lam^{\pm}_\eta\tau)+\s(\tau,y){\mp} o_\eta(1) & \hbox{in } \R^+\times\R^{N+1}\\
    V^{\pm}_\eta(0,Y)=0 & \hbox{on }\R^{N+1},
  \end{array}
\right.\end{equation}  where $0\leq o_\eta(1)\rightarrow0$ as
$\eta\rightarrow0^+$, such that \beq\label{appcorr1}
\LIM_{\eta\rightarrow0^+}\lam^+_\eta(p,L)=\LIM_{\eta\rightarrow0^+}\lam^-_\eta(p,L)=\lam(p,L),\eeq
locally uniformly in $(p,L)$, $\lam^{\pm}_\eta$ satisfy (i) and
(ii) of Proposition \ref{Hprop} and for any
$(\tau,Y)\in\R^+\times\R^{N+1}$
\beq\label{appcorr2}|V^{\pm}_{\eta}(\tau,Y)|\leq C.\eeq
 Moreover $V^{\pm}_{\eta}$ are of class $C^{2}$ w.r.t. $y_{N+1}$, and for any $0<\alpha<\min\{1,2s\}$ 
\beq\label{appcorr3}
-1\leq\p_{y_{N+1}}V^{\pm}_\eta \leq
\frac{\|W''\|_\infty}{\eta},\eeq 
\begin{equation}\label{contrderivappcorr}
\|\p^2_{y_{N+1}y_{N+1}}V^{\pm}_\eta\|_\infty\leq C_{\eta},\quad <\p_{y_{N+1}}V^{\pm}_\eta>_y^\al,\,\le C_{\eta,\alpha}.
\end{equation}
\end{prop}

\section{Proof of Theorem \ref{convergence}}\label{convprofs>1/2sec}

To prove  Theorem \ref{convergence}, as explained in Section \ref{ansatzsec}, 
we introduce the higher dimensional problem \eqref{Ueps>1/2} and we prove the convergence of the solution $U^\ep$ to the solution $U^0$  of  \eqref{Ueffetts>1/2}.
Let us first state the following
\begin{prop}\label{existUep}For $\ep>0$ there exists
$U^{\ep}\in C_b(\R^+\times\R^{N+1})$ (unique) viscosity solution
of \eqref{Ueps>1/2}. Moreover, there exists a constant $C>0$
independent of $\ep$ such that
\begin{equation}\label{stimathmexistUeps>1/2}|U^\ep(t,x,x_{N+1})-u_0(x)-x_{N+1}|\leq
Ct.\end{equation}
\end{prop}
Proposition \ref{existUep} as well as the existence of a unique solution of problems  \eqref{ueps>1/2}, \eqref{ueffetts>1/2} and  \eqref{Ueffetts>1/2} is a consequence of the Perron's method and the comparison principle for these equations, see \cite{mp} and references therein. 
Let us exhibit the link between the
problem in $\R^N$ and the problem in $\R^{N+1}$.
\begin{lem}[Link between the problems on $\R^N$ and on $\R^{N+1}$]\label{linkuepUep} If $u^\ep$ and $U^\ep$ denote respectively the solution
of \eqref{ueps>1/2} and \eqref{Ueps>1/2}, then we have
$$\left|U^\ep(t,x,x_{N+1})-u^\ep(t,x)-\ep\left\lfloor\frac{x_{N+1}}{\ep}\right\rfloor\right|\leq
\ep,$$
\beq\label{linklemmdis}U^\ep\left(t,x,x_{N+1}+
\ep\left\lfloor\frac{a}{\ep}\right\rfloor\right)
=U^\ep(t,x,x_{N+1})+\ep\left\lfloor\frac{a}{\ep}\right\rfloor\quad\text{for
any }a\in\R.\eeq
\end{lem}
This lemma follows from
the comparison principle for
 \eqref{Ueps>1/2} and the invariance by $\ep$-translations w.r.t. $x_{N+1}$.
\begin{lem}\label{linkuU} Let $u^0$ and $U^0$ be respectively the solutions of \eqref{ueffetts>1/2} and \eqref{Ueffetts>1/2}. Then, we have
$$U^0(t,x,x_{N+1})=u^0(t,x)+x_{N+1}.$$
\end{lem}

Lemma \ref{linkuU} is a consequence of the comparison principle for
\eqref{Ueffetts>1/2} and the invariance by translations w.r.t.
$x_{N+1}$.

Let us proceed with the proof of Theorem \ref{convergence}. In what follows we will use the notation $X=(x,x_{N+1})$. 
By \eqref{stimathmexistUeps>1/2}, we know that the
family of functions $\{U^\ep\}_{\ep>0}$ is locally bounded, then
$$ U^+(t,X):=\limsup_{\ep\rightarrow0}{}^* \ U^\epsilon(t,X):=
\limsup_{{\ep\rightarrow0}\atop{
(t',X')\rightarrow(t,X)
}}\ U^\epsilon(t',X')$$ is everywhere finite, so
it becomes classical
to prove that~$U^+$ is a subsolution of (\ref{Ueffetts>1/2}).

Similarly, we can prove that
$$U^-(t,X)
:={\liminf_{\ep\rightarrow0}} {}_* \ U^\epsilon(t,X):=
\liminf_{{\ep\rightarrow0}\atop{
(t',X')\rightarrow(t,X)
}}\ U^\epsilon(t',X')
$$ is a supersolution
of \eqref{Ueffetts>1/2}. Moreover
$U^+(0,X)=U^-(0,X)=u_0(x)+x_{N+1}$. The comparison
principle for \eqref{Ueffetts>1/2} 
 then  implies that $U^+\leq U^-$. Since the reverse inequality $U^-\leq U^+$ always holds true, we
conclude that the two functions coincide with $U^0$, the unique
viscosity solution of \eqref{Ueffetts>1/2}.

By Lemmata \ref{linkuepUep} and \ref{linkuU}, the convergence of
$U^\ep$ to $U^0$ proves in particular that $u^\ep$ converges
towards $u^0$ viscosity solution of \eqref{ueffetts>1/2}.
\bigskip

To prove that  $U^+$ is a subsolution of \eqref{Ueffetts>1/2}, we argue by contradiction.  We consider a test function $\phi$ such
that $U^+-\phi$ attains a zero maximum at $(t_0,X_0)$ with $t_0>0$
and $X_0=(x_0,x_{N+1}^0)$. Without loss of generality we may
assume that the maximum is strict and global. Suppose that there
exists $\theta>0$ such that
$$\p_t\phi(t_0,X_0)=\overline{H}_1(\nabla_x \phi(t_0,X_0))+\theta.$$
By Proposition \ref{Hprop}, we know that there exists $L_1>0$ (that we take minimal) such
that
$$\overline{H}_1(\nabla_x\phi(t_0,X_0))+\theta=\overline{H}(\nabla_x\phi(t_0,X_0),0)+\theta=\overline{H}(\nabla_x\phi(t_0,X_0),L_1).$$
By Propositions \ref{apprcorrectors} and \ref{Hprop}, we can
consider a sequence $L_\eta\rightarrow L_1$ as
$\eta\rightarrow0^+$, such that
$\lam^+_\eta(\nabla_x\phi(t_0,X_0),L_\eta)=\lam(\nabla_x\phi(t_0,X_0),L_1)$.
We choose $\eta$ so small that $L_\eta-o_\eta(1)\geq L_1/2>0$,
where $o_\eta(1)$ is defined in Proposition \ref{apprcorrectors}.
Let $V^+_{\eta}$ be the approximate supercorrector given by
Proposition \ref{apprcorrectors} with
$$ p=\nabla_x\phi(t_0,X_0),\quad
L=L_\eta$$ and
$$\lam^+_\eta=\lam^+_\eta(p,L_\eta)=\lam(p,L_1)= \p_t\phi(t_0,X_0).$$
For simplicity of notations, in the following we denote
$V=V^+_\eta$. We consider the function $$F(t,X)=\phi(t,X)-p\cdot
x-\lam^+_\eta t,$$ and as in \cite{mp}  we introduce the
``$x_{N+1}$-twisted perturbed test function'' $\phi^\epsilon$
defined by:

\begin{equation}\label{phiep}\phi^\epsilon(t,X):=
\begin{cases}
\phi(t,X)+\epsilon
V\left(\frac{t}{\epsilon},\frac{x}{\epsilon},\frac{F(t,X)}{\epsilon}\right)+\epsilon
k_\epsilon
 & \text{in}\quad (\frac{t_0}{2},2t_0)\times B_\frac{1}{2}(X_0)\\
U^\epsilon (t,X) &\text{outside},
\end{cases}
\end{equation}
where $k_\epsilon\in\Z$ will be chosen later.\\

We are going to
prove that $\phi^\epsilon$ is a supersolution of \eqref{Ueps>1/2} in
$Q_{r,r}(t_0,X_0)$ for some $r<\frac{1}{2}$ properly chosen and
such that $Q_{r,r}(t_0,X_0)\subset(\frac{t_0}{2},2t_0)\times
B_\frac{1}{2}(X_0)$. First, we observe that since $U^+-\phi$ attains a
strict maximum at $(t_0,X_0)$ with $U^+-\phi=0$ at $(t_0,X_0)$ and
$V$ is bounded, we can ensure that there exists $\ep_0=\ep_0(r)>0$
such that for $\ep\leq \ep_0$
\begin{equation}\label{phiep2}U^\ep(t,X)\leq \phi(t,X)+\epsilon
V\left(\frac{t}{\epsilon},\frac{x}{\epsilon},\frac{F(t,X)}{\epsilon}\right)-\gamma_r,\quad\text{in
}\left(\frac{t_0}{3},3t_0\right)\times B_1(X_0)\setminus
Q_{r,r}(t_0,X_0)\end{equation} for some $\gamma_r=o_r(1)>0$. Hence
choosing $k_\ep=\lceil \frac{-\gamma_r}{\ep}\rceil$ we get

\beqs U^\ep\leq \phi^\ep\quad \text{outside }Q_{r,r}(t_0,X_0).\eeqs

Let us next study the equation satisfied by~$\phi^\ep$. 
For this, we observe that
$$ \frac{a}{\ep} -1 \le
\left\lfloor\frac{a}{\ep}\right\rfloor \le \frac{a}{\ep}$$
and so, from \eqref{linklemmdis}, we
deduce that 
$$ 
U^\ep(t,x,x_{N+1}) +a-\ep \ \le \
U^\ep\left(t,x,x_{N+1}+
\ep\left\lfloor\frac{a}{\ep}\right\rfloor\right)
\ \le \ U^\ep(t,x,x_{N+1}) + a.$$
Consequently, passing to the limit, we obtain that
$U^+(t,x,x_{N+1}+a)=U^+(t,x,x_{N+1})+a$ for any
$a\in\R$.

{F}rom this, we derive that $\p_{x_{N+1}}
F(t_0,X_0)=\p_{x_{N+1}}\phi(t_0,X_0)=1$. Then, there exists
$r_0>0$ such that the map

$$\begin{array}{cccc}
                Id\times F:&Q_{r_0,r_0}(t_0,X_0) & \longrightarrow & \mathcal{U}_{r_0} \\
                &(t,x,x_{N+1})  & \longmapsto & (t,x,F(t,x,x_{N+1})) \\
\end{array}$$
is a $C^1$-diffeomorphism from $Q_{r_0,r_0}(t_0,X_0)$ onto its
range $\mathcal{U}_{r_0}$. Let $G:\mathcal{U}_{r_0}\rightarrow\R$
be the map such that $$\begin{array}{cccc}
                Id\times G:&\mathcal{U}_{r_0} & \longrightarrow & Q_{r_0,r_0}(t_0,X_0)  \\
                &(t,x,\xi_{N+1})  & \longmapsto & (t,x,G(t,x,\xi_{N+1})) \\
\end{array}$$ is the inverse of $Id\times F$.
 Let
us introduce the variables $\tau=t/\ep$, $Y=(y,y_{N+1})$ with
$y=x/\ep$ and $y_{N+1}=F(t,X)/\ep$. Let us consider a test
function $\psi$ such that $\phi^\ep-\psi$ attains a global zero
minimum at $(\ts,\Xs)\in Q_{r_0,r_0}(t_0,X_0)$ and define \beqs
\Gamma^\ep(\tau,Y)=\frac{1}{\ep}[\psi(\ep\tau,\ep y, G(\ep\tau,\ep
y,\ep y_{N+1}))-\phi(\ep\tau,\ep y,G(\ep\tau,\ep y,\ep
y_{N+1}))]-k_\ep.\eeqs Then
$$\psi(t,X)=\phi(t,X)+\epsilon
\Gamma^\ep\left(\frac{t}{\epsilon},\frac{x}{\epsilon},\frac{F(t,X)}{\epsilon}\right)+\epsilon
k_\epsilon$$ and $\Gamma^\ep$ is a test function for $V$:

 \beq\label{gammaeptestv}
\Gamma^\ep(\tas,\Ys)=V(\tas,\Ys)\quad\text{and}\quad
\Gamma^\ep(\tau,Y)\leq V(\tau,Y)\quad \text{for all }(\ep\tau,\ep
Y)\in Q_{r_0,r_0}(t_0,X_0),\eeq where $\tas=\ts/\ep$,
$\ys=\xs/\ep,$ $\overline{y}_{N+1}=F(\ts,\Xs)/\ep$,
$\Ys=(\ys,\overline{y}_{N+1})$. From Proposition
\ref{apprcorrectors}, we know that $V$ is Lipschitz continuous
w.r.t. $y_{N+1}$ with Lipschitz constant $M_\eta$ depending on
$\eta$. This implies that
\beq\label{vlipyn+1}|\p_{y_{N+1}}\Gamma^\ep(\tas,\Ys)|\leq
M_\eta.\eeq

Simple computations yield with $P=(p,1)\in\R^{N+1}$:
\beq\label{equaprogfconvchangevar}\left\{%
\begin{array}{ll}
    \lam^+_\eta+\p_{\tau}\Gamma^\ep(\tas,\Ys)=\p_t\psi(\ts,\Xs)+\left(1+\p_{y_{N+1}}\Gamma^\ep(\tas,\Ys)\right)
    (\p_t\phi(t_0,X_0)-\p_t\phi(\ts,\Xs)), \\
    \lam^+_\eta \tas+P\cdot\Ys+V(\tas,\Ys)=\frac{\phi^\ep(\ts,\Xs)}{\ep}-k_\ep. \\
\end{array}%
\right.\eeq Using \eqref{equaprogfconvchangevar} and
\eqref{vlipyn+1}, equation \eqref{apprcorrequ} yields for any
$\rho>0$
\begin{equation}\label{phiepquat}\begin{split}\p_t\psi(\ts,\Xs)+o_r(1)&\geq
L_\eta+\I^{1,\rho}[\Gamma^\ep(\tas,\cdot,\overline{y}_{N+1}),\ys]+\I^{2,\rho}[V(\tas,\cdot,\overline{y}_{N+1}),\ys]
\\&-W'\left(\frac{\phi^\ep(\ts,\Xs)}{\ep}\right)+\sigma\left(\frac{\ts}{\ep},\frac{\xs}{\ep}\right)-o_\eta(1).\end{split}\end{equation}
Now, to complete the proof of Theorem \ref{convergence},
we state
the following lemma (which will be proved in the next subsection):

\begin{lem} {\bf (Supersolution property for $\phi^\ep$)}\label{converglem}\\
For $\ep\leq \ep_0(r)< r\leq r_0$, we
have \beq\begin{split}\label{supphieplemm}\p_t\psi(\ts,\Xs)&\geq
\ep^{2s-1}\left(\I^{1,1}\left[\psi(\ts,\cdot,\overline{x}_{N+1}),\xs\right]+\I^{2,1}\left[\phi^\ep(\ts,\cdot,\overline{x}_{N+1}),\xs\right]\right)\\&
-W'\left(\frac{\phi^\ep(\ts,\Xs)}{\ep}\right)+\sigma\left(\frac{\ts}{\ep},\frac{\xs}{\ep}\right)-o_\eta(1)+o_r(1)+L_\eta.\end{split}\eeq\end{lem}

The proof of Lemma~\ref{converglem} is postponed
to the next subsection, for the convenience of the reader,
so we complete now the proof of Theorem \ref{convergence}.
For this, let $r\leq r_0$ be so small that $o_r(1)\geq-L_1/4$. Then,
recalling that $L_\eta-o_\eta(1)\geq L_1/2$, for $\ep\leq
\ep_0(r)$ we have
\begin{equation*}\begin{split}\p_t\psi(\ts,\Xs)&\geq
\ep^{2s-1}\left(\I^{1,1}\left[\psi(\ts,\cdot,\overline{x}_{N+1}),\xs\right]+\I^{2,1}\left[\phi^\ep(\ts,\cdot,\overline{x}_{N+1}),\xs\right]\right)
-W'\left(\frac{\phi^\ep(\ts,\Xs)}{\ep}\right)\\&+\sigma\left(\frac{\ts}{\ep},\frac{\xs}{\ep}\right)+\frac{L_1}{4},\end{split}\end{equation*}
and therefore $\phi^\ep$ is a supersolution of \eqref{Ueps>1/2} in
$Q_{r,r}(t_0,X_0)$.\\
Since $U^\ep\leq \phi^\ep$ outside
$Q_{r,r}(t_0,X_0)$, by the comparison principle, we conclude that
 $$U^\ep(t,X)\leq
\phi(t,X)+\epsilon
V\left(\frac{t}{\epsilon},\frac{x}{\epsilon},\frac{F(t,X)}{\epsilon}\right)+\ep
k_\ep\quad\text{in }Q_{r,r}(t_0,X_0)$$ and we obtain the desired
contradiction by passing to the upper limit as $\ep\rightarrow 0$ at $(t_0,X_0)$ using the fact that $U^+(t_0,X_0)=\phi(t_0,X_0)$: $0\leq-\gamma_r$.\\
This ends the proof of Theorem \ref{convergence}.\\

\subsection{Proof of Lemma \ref{converglem}}
The result will follow from (\ref{phiepquat})
and the following inequality
\begin{equation}\label{eq::s15}
\begin{array}{l}
\I^{1,\rho}[\Gamma^\ep(\tas,\cdot,\overline{y}_{N+1}),\ys]+\I^{2,\rho}[V(\tas,\cdot,\overline{y}_{N+1}),\ys]\\ 
\\
\ge \ep^{2s-1}\left(\I^{1,1}\left[\psi(\ts,\cdot,\overline{x}_{N+1}),\xs\right]+\I^{2,1}\left[\phi^\ep(\ts,\cdot,\overline{x}_{N+1}),\xs\right]\right) + o_r(1)
\end{array}
\end{equation}

Keep in mind that $\ys_{N+1}=\frac{F(\ts,\Xs)}{\ep}$. Since
$\psi(t,X)=\phi(t,X)+\ep\Gamma^\ep\left(\frac{t}{\ep},\frac{x}{\ep},\frac{F(t,X)}{\ep}\right)+\ep
k_\ep$, we have
\begin{equation}\label{i1psi}\begin{split}\I^{1,1}\left[\psi(\ts,\cdot,\overline{x}_{N+1}),\xs\right]= I_1
+I_2,
\end{split}
\end{equation}where
$$\left\{\begin{array}{lll}
I_1 &=&\displaystyle \int_{|x|\leq1}\ep
\left(\begin{array}{l}
\Gamma^\ep\left(\frac{\overline{t}}{\ep},\frac{\xs+x}{\ep},
\frac{F(\ts,\xs+x,\overline{x}_{N+1})}{\ep}\right)-\Gamma^\ep(\tas,\Ys)\\
-\nabla_{y}\Gamma^\ep(\tas,\Ys)\cdot
\frac{x}{\ep}-\p_{y_{N+1}}\Gamma^\ep(\tas,\Ys)\nabla_x
F(\ts,\Xs)\cdot \frac{x}{\ep}
\end{array}\right)\mu(dx),\\
&&\\
I_2& =&\displaystyle \int_{|x|\leq 1}\left(\phi(\ts,\xs+x,\overline{x}_{N+1})-\phi(\ts,\Xs)-\nabla\phi(\ts,\Xs)\cdot
x\right)\mu(dx).
\end{array}\right.$$

In order to show (\ref{eq::s15}), we show successively in Steps 1, 2 and 3:
$$\left\{\begin{array}{l}
\ep^{2s-1}I_1\leq
\I^{1,\rho}[\Gamma^\ep(\tas,\cdot,\overline{y}_{N+1}),\ys]+\I^{2,\rho}[V(\tas,\cdot,\overline{y}_{N+1}),\ys]
+o_r(1)+C_\ep\rho^{2-2s}\\
\\
\ep^{2s-1}I_2\leq  o_r(1)\\
\\
\ep^{2s-1}\I^{2,1}\left[\phi^\ep(\ts,\cdot,\overline{x}_{N+1}),\xs\right]\leq  o_r(1)
\end{array}\right.$$
Because the expressions are non linear and non-local and with a singular kernel, 
there is no simple computation and we have to carefully check those inequalities 
sometimes splitting terms in easier parts to estimate.
\medskip

\noindent{\bf Step 1:} We can choose $\ep_0$ so small that for any
$\ep\leq\ep_0$ and any $\rho>0$ small enough $$\ep^{2s-1}I_1\leq
\I^{1,\rho}[\Gamma^\ep(\tas,\cdot,\overline{y}_{N+1}),\ys]+\I^{2,\rho}[V(\tas,\cdot,\overline{y}_{N+1}),\ys]
+o_r(1)+C_\ep\rho^{2-2s}.$$ \\

Take $\rho>0$,  $\delta>\rho$ small and $R>0$ large and such that
$\ep R<1$. Since $g$ is even, we can write

$$I_1=I_1^0+I_1^1+I_1^2+I_1^3,$$ where
\begin{equation*}\begin{split}I_1^0&=\int_{|x|\leq\ep\rho}\ep\left(\Gamma^\ep\left(\frac{\overline{t}}{\ep},\frac{\xs+x}{\ep},
\frac{F(\ts,\xs+x,\overline{x}_{N+1})}{\ep}\right)-\Gamma^\ep(\tas,\Ys)-\nabla_{y}\Gamma^\ep(\tas,\Ys)\cdot
\frac{x}{\ep}\right.\\&\left.-\p_{y_{N+1}}\Gamma^\ep(\tas,\Ys)\nabla_x
F(\ts,\Xs)\cdot \frac{x}{\ep}\right)\mu(dx),\end{split}
\end{equation*}
$$I_1^1=\int_{\ep\rho\leq|x|\leq\ep \delta}\ep\left(\Gamma^\ep\left(\frac{\overline{t}}{\ep},\frac{\xs+x}{\ep},
\frac{F(\ts,\xs+x,\overline{x}_{N+1})}{\ep}\right)-
\Gamma^\ep(\tas,\Ys)\right)\mu(dx),$$
$$I_1^2=\int_{\ep\delta\leq|x|\leq\ep R}\ep\left(\Gamma^\ep\left(\frac{\overline{t}}{\ep},\frac{\xs+x}{\ep},
\frac{F(\ts,\xs+x,\overline{x}_{N+1})}{\ep}\right)-
\Gamma^\ep(\tas,\Ys)\right)\mu(dx),$$
$$I_1^3=\int_{\ep R\leq|x|\leq1}\ep\left(\Gamma^\ep\left(\frac{\overline{t}}{\ep},\frac{\xs+x}{\ep},
\frac{F(\ts,\xs+x,\overline{x}_{N+1})}{\ep}\right)-
\Gamma^\ep(\tas,\Ys)\right)\mu(dx).$$

 Moreover
$$\I^{2,\rho}[V(\tas,\cdot,\overline{y}_{N+1}),\ys]=J_1+J_2+J_3,$$ where

$$J_1=\int_{\rho<|z|\leq\delta}(V(\tas,\ys+z,\overline{y}_{N+1})- V(\tas,\Ys))\mu(dz),$$
$$J_2=\int_{\delta<|z|\leq R}(V(\tas,\ys+z,\overline{y}_{N+1})- V(\tas,\Ys))\mu(dz),$$

$$J_3=\int_{|z|>R}(V(\tas,\ys+z,\overline{y}_{N+1})- V(\tas,\Ys))\mu(dz).$$
\medskip

\noindent STEP 1.1: \emph{Estimate of $\ep^{2s-1}I_1^0$ and
$\I^{1,\rho}[\Gamma^\ep(\tas,\cdot,\overline{y}_{N+1}),\ys]$.}

Since $\Gamma^\ep$ is of class $C^2$, we have 
\beq\label{step1.1}
|\ep^{2s-1}I_1^0|,\,|\I^{1,\rho}[\Gamma^\ep(\tas,\cdot,\overline{y}_{N+1}),\ys]|\leq
C_\ep\rho^{2-2s},\eeq where $C_\ep$ depends on the second derivatives of
$\Gamma^\ep$. Notice
that if we knew that $V$ is smooth in $y$ too, we could choose $\rho=0$. 
\medskip

\noindent STEP 1.2 \emph{ Estimate of $\ep^{2s-1}I_1^1-J_1$. }

Using \eqref{gammaeptestv} and the fact that $g$ is even, we can
estimate $\ep^{2s-1}I_1^1-J_1$ as follows
\begin{equation*}\begin{split}\ep^{2s-1}I_1^1-J_1&\leq\int_{\rho<|z|\leq\delta}\left[V\left(\tas,\ys+z,\frac{F(\ts,\xs+\ep
z,\xs_{N+1})}{\ep}\right)-V\left(\tas,\ys+z,\frac{F(\ts,\Xs)}{\ep}\right)\right]\mu(dz)\\&
=\int_{\rho<|z|\leq\delta}\left\{\left[V\left(\tas,\ys+z,\frac{F(\ts,\xs+\ep
z,\xs_{N+1})}{\ep}\right)-V\left(\tas,\ys+z,\frac{F(\ts,\Xs)}{\ep}\right)\right.\right.\\&\left.
-\p_{y_{N+1}}V\left(\tas,\ys+z,\frac{F(\ts,\Xs)}{\ep}\right)\nabla_xF(\ts,\Xs)\cdot
z\right]\\&
\left.+\left[\p_{y_{N+1}}V(\tas,\ys+z,\ys_{N+1})-\p_{y_{N+1}}V(\tas,\Ys)\right]\nabla_x
F(\ts,\Xs)\cdot z\right\}\mu(dz).
\end{split}
\end{equation*}
Next, using  \eqref{contrderivappcorr}, we get
\beq\label{i1-j1}\ep^{2s-1}I_1^1-J_1\leq
C\int_{|z|\leq\delta}(|z|^2+|z|^{1+\al})\mu(dz)\leq C
\delta^{\al+1-2s},\eeq
for $2s-1<\al<1$.
\medskip

\noindent STEP 1.3 \emph{ Estimate of $\ep^{2s-1}I_1^2-J_2$. }

If $M_\eta$ is the Lipschitz constant of $V$ w.r.t. $y_{N+1}$,
then
\begin{equation*}\begin{split}\ep^{2s-1}I_1^2-J_2&\leq\int_{\delta<|z|\leq
R}\left(V\left(\tas,\ys+z,\frac{F(\ts,\xs+\ep
z,\xs_{N+1})}{\ep}\right)-V\left(\tas,\ys+z,\frac{F(\ts,\Xs)}{\ep}\right)\right)\mu(dz)\\
&\leq M_\eta\int_{\delta<|z|\leq R}\left|\frac{F(\ts,\xs+\ep
z,\xs_{N+1})}{\ep}-\frac{F(\ts,\Xs)}{\ep}\right|\mu(dz)\\&\leq
M_\eta\int_{\delta<|z|\leq R}\sup_{|z|\leq
R}|\nabla_xF(\ts,\xs+\ep z,\xs_{N+1})||z|\mu(dz).
\end{split}
\end{equation*}
Then 

\beq\label{i2-j2}\ep^{2s-1}I_1^2-J_2\leq C\sup_{|z|\leq
R}|\nabla_xF(\ts,\xs+\ep z,\xs_{N+1})|\left(\frac{1}{\delta^{2s-1}}-\frac{1}{R^{2s-1}}\right).\eeq
\medskip

\noindent STEP 1.4: {\em Estimate of $\ep^{2s-1}I_1^3$ and $J_3$. }

Since $V$ is uniformly bounded on $\R^+\times\R^{N+1}$, we have
\begin{equation}\label{i3est}\begin{split}\ep^{2s-1}I_1^3&\leq
\int_{R<|z|\leq
\frac{1}{\ep}}\left(V\left(\tas,\ys+z,\frac{F(\ts,\xs+\ep
z,\xs_{N+1})}{\ep}\right)- V(\tas,\Ys)\right)\mu(dz)\\&\leq
\int_{|z|>R}2\|V\|_\infty\mu(dz)\leq \frac{C}{R^{2s}}.
\end{split}
\end{equation}
Similarly \beq\label{jeest}|J_3|\leq \frac{C}{R^{2s}}.\eeq

Now, from \eqref{step1.1}, \eqref{i1-j1}, \eqref{i2-j2},
\eqref{i3est} and \eqref{jeest}, we infer that \beqs\begin{split}
\ep^{2s-1}I_1&\leq
\I^{1,\rho}[\Gamma^\ep(\tas,\cdot,\overline{y}_{N+1}),\ys]+\I^{2,\rho}[V(\tas,\cdot,\overline{y}_{N+1}),\ys]+
2C_\ep\rho^{2-2s}+C \delta^{\al+1-2s}\\&+C\sup_{|z|\leq R}|\nabla_xF(\ts,\xs+\ep
z,\xs_{N+1})|\left(\frac{1}{\delta^{2s-1}}-\frac{1}{R^{2s-1}}\right)+\frac{C}{R^{2s}}.\end{split}\eeqs
We remark that, from the definition of $F$, we have
\beqs\begin{split} \sup_{|z|\leq R}|\nabla_xF(\ts,\xs+\ep z,\xs_{N+1})|&\leq \sup_{|z|\leq R}|\nabla\phi(\ts,\xs+\ep z,\xs_{N+1})-\nabla\phi(t_0,X_0)|
\\&\leq\sup_{|z|\leq R}|\nabla\phi(\ts,\xs+\ep z,\xs_{N+1})-\nabla\phi(\ts,\Xs)|\\&+|\nabla\phi(\ts,\Xs)-\nabla\phi(t_0,X_0)|
\\&\leq C(\ep R+r).\end{split}\eeqs
Now, we choose $R=R(r)$ such $R\rightarrow+\infty$ as
$r\rightarrow0^+$, $\ep_0=\ep_0(r)$ such that $R\ep_0(r)\leq r$
and $\delta=\delta(r)>0$ such that $\delta\rightarrow0$ as
$r\rightarrow0^+$ and $r/\delta^{2s-1}\rightarrow0$ as
$r\rightarrow0^+$. With this choice, for any $\ep\leq\ep_0$ and
any $\rho<\delta$
$$C \delta^{\al+1-2s}+C\sup_{|z|\leq
R}|\nabla_xF(\ts,\xs+\ep
z,\xs_{N+1})|\left(\frac{1}{\delta^{2s-1}}-\frac{1}{R^{2s-1}}\right)+\frac{C}{R^{2s}}=o_r(1)\quad\text{as
}r\rightarrow0^+,$$ and Step 1 is proved.
\medskip

The next two steps are trivial.
\medskip

\noindent{\bf Step 2: $\ep^{2s-1}I_2\leq  C\ep^{2s-1}$. }
\medskip

\noindent{\bf Step 3: $\ep^{2s-1}\I^{2,1}\left[\phi^\ep(\ts,\cdot,\overline{x}_{N+1}),\xs\right]\leq C\ep^{2s-1}$. }\medskip

Finally  Steps 1, 2 and 3 give
\beqs\begin{split}
&\ep^{2s-1}\I^{1,1}\left[\psi(\ts,\cdot,\overline{x}_{N+1}),\xs\right]+\ep^{2s-1}\I^{2,1}\left[\phi^\ep(\ts,\cdot,\overline{x}_{N+1}),\xs\right]\\&\leq
\I^{1,\rho}[\Gamma^\ep(\tas,\cdot,\overline{y}_{N+1}),\ys]+\I^{2,\rho}[V(\tas,\cdot,\overline{y}_{N+1}),\ys]
+o_r(1)+C_\ep\rho^{2-2s}.\end{split}\end{equation*} from which, using
inequality \eqref{phiepquat} and letting $\rho\rightarrow0^+$, we
get \eqref{supphieplemm}.


\section{Proof of Theorem \ref{convergences<1/2}}\label{yuio9}

The proof of Theorem \ref{convergences<1/2} is similar to the proof of Theorem \ref{convergence}, therefore we only give a sketch of it.
As in Theorem \ref{convergence}, we argue by contradiction, assuming that there is a 
test function $\phi$ such
that $U^+-\phi$ attains a strict zero maximum at $(t_0,X_0)$ with $t_0>0$
and $X_0=(x_0,x_{N+1}^0)$, and
$$\p_t \phi(t_0,X_0)=\overline{H}_2(L_0)+\theta$$
for some $\theta>0$, where  
\begin{equation}\label{l0}\begin{split}L_0=&\int_{|x|\leq1}(\phi(t_0,x_0+x,x_{N+1}^0)-\phi(t_0,X_0)-\nabla_x\phi(t_0,X_0)\cdot
x)\mu(dx)\\&+\int_{|x|>1}(U^+(t_0,x_0+x,x_{N+1}^0)-U^+(t_0,X_0))\mu(dx).\end{split}
\end{equation}
Then, we choose $L_1>0$ and a sequence 
 $L_\eta\rightarrow L_1$ as
$\eta\rightarrow0^+$, such that
$$\lam^+_\eta(0,L_\eta+L_0)=\lam(0,L_1+L_0)=\lam(0,L_0)+\theta=\overline{H}_2(L_0)+\theta.$$
Let $V$ be the approximate supercorrector given by
Proposition \ref{apprcorrectors} with
$$ p=0,\quad
L=L_0+L_\eta$$ and
$$\lam^+_\eta=\lam^+_\eta(0,L_0+L_\eta)=\p_t \phi(t_0,X_0).$$
Let us introduce 
the
``$x_{N+1}$-twisted perturbed test function'' $\phi^\epsilon$
defined by:

\begin{equation*}\phi^\epsilon(t,X):=
\begin{cases}
\phi(t,X)+\epsilon^{2s}
V\left(\frac{t}{\epsilon^{2s}},\frac{x}{\epsilon},\frac{F(t,X)}{\epsilon^{2s}}\right)+\epsilon^{2s}
k_\epsilon
 & \text{in}\quad (\frac{t_0}{2},2t_0)\times B_\frac{1}{2}(X_0)\\
U^\epsilon (t,X) &\text{outside},
\end{cases}
\end{equation*} where $F(t,X)=\phi(t,X)-\lam^+_\eta t$ and 
 $k_\epsilon\in\Z$ is opportunely chosen.
As in Section \ref{convprofs>1/2sec}, we can prove that 
$\phi^\ep$ is a supersolution of \eqref{Ueps<1/2} in a neighborhood $Q_{r,r}(t_0,X_0)$ of $(t_0,X_0)$, for some small $r$ properly chosen.  Moreover
\beqs U^\ep\leq \phi^\ep\quad \text{outside }Q_{r,r}(t_0,X_0).\eeqs

The contradiction follows by comparison.


\section{ Proof of Theorem \ref{hullprop}}\label{OR}
In this section we restrict ourself to the case: $N=1$, $\I=-(-\Delta)^s$ and $\sigma\equiv 0$.  For fixed $p,L\in\R$, let us introduce the corrector 
\beqs u(\tau,y):=w(\tau,y)+py\eeqs where $w$ is the solution of \eqref{w} given by Theorem \ref{ergodic}. Then $u$ is solution of 

\begin{equation}\label{sv}
\begin{cases}
\p_{\tau} u= L+\I[u(\tau,\cdot)]-W'(u)&\text{in}\quad \R^+\times\R\\ 
u(0,y)=py& \text{on}\quad \R,
\end{cases}
\end{equation}
and by the ergodic property \eqref{wergodic} it satisfies
\beq\label{uergodic} |u(\tau,y)-py-\lambda\tau|\leq C.\eeq

The idea underlying the proof of Theorem \ref{hullprop} is related to a fine asymptotics of equation~\eqref{sv}. We want to show that if $u$ solves \eqref{sv}
with $p=\delta|p_0|$ and $L=\delta^{2s}L_0$, i.e.
\begin{equation}\label{eq::s102}
\partial_\tau u = \delta^{2s} L_0 +\I[u(\tau,\cdot)] -W'(u)
\end{equation}
and $u(0,y)=\delta p_0y$, then 
\beqs u(\tau,y)\sim \delta p_0y+\lambda\tau+\text{bounded}\quad\text{with }\lambda\sim \delta^{1+2s}c_0|p_0|L_0.\eeqs
We deduce that we should have 
\beqs\frac{u(\tau,y)}{\tau}\rightarrow \lam=\delta^{1+2s}c_0|p_0|L_0\quad\text{as }\tau\rightarrow+\infty.\eeqs
We see that this $\lam=\overline{H}(\delta p_0,\delta^{2s}L_0)$ is exactly the one we expect asymptotically in Theorem~\ref{hullprop}.

Following the idea of  \cite{mp2}, one may expect to find particular solutions $u$ of (\ref{eq::s102}) that we can write
$$u(\tau,y)=h(\delta p_0 y + \lambda \tau)$$
for some $\lambda\in\R$ and 
a function $h$ (called hull function) satisfying
$$|h(z)-z|\le C.$$
This means that $h$ solves
$$\lambda h' = \delta^{2s} L_0+ \delta^{2s} |p_0|^{2s} \I [h] -W'(h).$$
Then it is natural to introduce the non-linear operator:
\begin{equation}\label{eq::s103}
NL^\lambda_{L_0}[h]:= \lambda h' - \delta^{2s} L_0- \delta^{2s} |p_0|^{2s} \I [h] + W'(h)
\end{equation}
and for the ansatz for $\lambda$:
$$\overline{\lambda}^{L_0}_\delta = \delta^{1+2s} c_0 |p_0|L_0$$
it is natural to look for an ansatz $h^{L_0}_\delta$ for $h$.
We define (see Proposition \ref{hproperties})
$$h_\delta^{L_0}(x)=\LIM_{n\rightarrow+\infty}s_{\delta,n}^{L_0}(x)$$
where for $s\ge \frac{1}{2}$ and for all
$p_0\neq 0$, $L_0\in\R$, $\delta>0$ and $n\in\N$ we
define the sequence of functions $\{s_{\delta,n}^{L_0}(x)\}_n$ by
\begin{equation}\label{eq::s-s100}
s_{\delta,n}^{L_0}(x)=\frac{\delta^{2s} L_0}{\al}+\SUM_{i=-n}^{n}\phi\left(\frac{x-i}{\delta|p_0|}\right)-n
+\delta^{2s}\SUM_{i=-n}^{n}\psi\left(\frac{x-i}{\delta|p_0|}\right)
\end{equation}
where $\al=W''(0)>0$ and $\phi$ is the solution of \eqref{phi}. 
The corrector $\psi$ is the solution of the following
problem

\begin{equation}\label{psi}
\begin{cases}\I[\psi]=W''(\phi)\psi+\frac{L_0}{W''(0)}(W''(\phi)-W''(0))+c\phi'&\text{in}\quad \R\\
\LIM_{x\rightarrow{+\atop -}\infty}\psi(x)=0\\
c=\frac{L_0}{\int_{\R} (\phi')^2}.
\end{cases}
\end{equation} 
For $s<\frac{1}{2}$, the function $\psi$ defined above may not decay fast enough so that the sequence 
$$\SUM_{i=-n}^{n}\psi\left(\frac{x-i}{\delta|p_0|}\right)$$ converges.  Therefore, in this case we define

\begin{equation}\label{eq::s-s100s<1/2}
s_{\delta,n}^{L_0}(x)=\frac{\delta^{2s} L_0}{\al}+\SUM_{i=-n}^{n}\phi\left(\frac{x-i}{\delta|p_0|}\right)-n
+\delta^{2s}\SUM_{i=-n}^{n}\psi\left(\frac{x-i}{\delta|p_0|}\right)\tau\left(\frac{x-i}{\delta|p_0|}\right)
\end{equation}
where $\tau=\tau_R$, is a smooth function satisfying 
\beq\label{tauproper}\begin{cases}0\leq\tau(x)\leq1&\text{for any  }x\in\R\\ \tau_R(x)=1&\text{if } |x|\leq R\\
\tau_R(x)=0& \text{if }|x|\geq 2 R.
\end{cases}
\eeq
The number  $R$ is a large parameter that will be chosen depending on $\delta$. 
\begin{prop}{ \bf (Good ansatz)}\label{hproperties}\\
Assume \eqref{W} and $R=\frac{1}{2\delta|p_0|}$ in \eqref{tauproper}. Then, for any $L\in\R$, $\delta>0$ and $x\in\R$, there exists
the finite limit
$$h_\delta^{L}(x)=\LIM_{n\rightarrow+\infty}s_{\delta,n}^{L}(x).$$
Moreover $h_\delta^{L}$ has the following properties:
\begin{itemize}
    \item [(i)]$h_\delta^{L}\in C^{2}(\R)$ and
satisfies
\begin{equation}\label{nl=odelta}NL_{L}^{\overline{\lam}_\delta^{L}}[h_\delta^{L}](x)=o(\delta^{2s}),\end{equation}
where
 $\LIM_{\delta\rightarrow 0}\frac{o(\delta^{2s})}{\delta^{2s}}=0$, uniformly
 for $x\in\R$ and locally uniformly in $L\in\R$;
Here 
$$\overline{\lam}_\delta^{L} = \delta^{1+2s} c_0 |p_0| L $$
and $NL_{L}^{\lambda}$ is defined in (\ref{eq::s103}).
    \item [(ii)]There exists a constant $C>0$ such that
    $|h_\delta^{L}(x)-x|\leq C$ for any $x\in\R$.
\end{itemize}
\end{prop}

\subsection{Proof of Theorem \ref{hullprop}}
We will show that Theorem \ref{hullprop} follows from Proposition~\ref{hproperties}
and the comparison principle.

Fix $\eta>0$ and let $L=L_0-\eta$. By (i) of Proposition
\ref{hproperties}, there exists $\delta_0=\delta_0(\eta)>0$ such
that for any $\delta\in(0, \delta_0)$ we have
\begin{equation}\label{NLhL_0}
NL_{L_0}^{\overline{\lam}_\delta^{L}}[h_\delta^L]=NL_{L}^{\overline{\lam}_\delta^{L}}[h_\delta^L]-\delta^{2s}\eta<0\quad\text{in
}\R.
\end{equation}
Let us consider the function $\widetilde{u}(\tau,y)$, defined by
$$\widetilde{u}(\tau,y) = h_\delta^L(\delta p_0 y + \overline{\lambda}^L_\delta \tau).$$
By (ii) of Proposition \ref{hproperties}, we have
\begin{equation}\label{propubarra}|\widetilde{u}(\tau,y)-\delta p_0 y - \overline{\lam}_\delta^L\tau|\leq \lceil
C\rceil,\end{equation} 
{ where $\lceil C\rceil$ is the ceil integer part of $C$.}
Moreover, by \eqref{NLhL_0} and \eqref{propubarra}, $\widetilde{u}$
satisfies
$$\left\{%
\begin{array}{ll}
   \widetilde{u}_\tau\leq \delta^{2s} L_0 + \I[\widetilde{u}]-W'(\widetilde{u})& \hbox{in } \R^+\times\R\\
    \widetilde{u}(0,y) \leq \delta p_0 y+ \lceil C\rceil & \hbox{on } \R.\\
\end{array}%
\right.$$
 Let $u(\tau,y)$ be the solution of \eqref{sv}, { with $p=\delta p_0$ and $L=\delta^{2s} L_0$}, whose existence is ensured by Theorem
\ref{ergodic}.
 Then from the comparison principle and the
periodicity of $W$, we deduce that
$$\widetilde{u}(\tau,y)\leq u(\tau,y)+\lceil C\rceil.$$
By the previous inequality and \eqref{propubarra}, we
get
$$\overline{\lam}_\delta^L\tau\leq u(\tau,y)-\delta p_0 y +2\lceil
C\rceil,$$ and dividing by $\tau$ and letting $\tau$ go to
$+\infty$, we finally obtain
$$\delta^{1+2s}c_0|p_0|(L_0-\eta)=\overline{\lam}_\delta^L\leq \overline{H}(\delta p_0,\delta^{2s}
L_0).$$Similarly, it is possible to show that
$$\overline{H}(\delta p_0,\delta^{2s}
L_0)\leq \delta^{1+2s}c_0|p_0|(L_0+\eta).$$We have proved that for any
$\eta>0$ there exists $\delta_0=\delta_0(\eta)>0$ such that for
any $\delta\in(0,\delta_0)$ we have
$$\left|\frac{\overline{H}(\delta p_0,\delta^{2s}
L_0)}{\delta^{1+2s}}-c_0|p_0|L_0\right|\leq c_0|p_0|\eta,$$ i.e.
\eqref{orowan}, as desired.

\subsection{Preliminary results}\mbox{ } \bigskip

Under the assumptions  (\ref{W}) on $W$, there exists a unique solution of \eqref{phi}
which is of class  $C^{2,\beta}$, as shown in  \cite{cs}, see also \cite{psv}. 
When $s<\frac{1}{2}$ we suppose in addition that $W$ is even. This implies that the function 
$$\phi-\frac{1}{2}$$ is odd.
The existence of
a solution of class $C^{1,\beta}_{loc}(\R)\cap L^\infty(\R)$ of the problem \eqref{psi} is
proved  in \cite{psv}.
Actually, the regularity of $W$ implies that $\phi\in
C^{4,\beta}(\R)$ and $\psi\in C^{3,\beta}(\R)$.

To prove Proposition \ref{hproperties} we need several preliminary
results. We first state the following two lemmata about the
behavior of the functions $\phi$ and $\psi$ at infinity. We denote
by $H(x)$ the Heaviside function defined by
\begin{equation*}
H(x)=\begin{cases}1 &\text{for }x\geq 0 \\ 0& \text{for
}x<0.\end{cases}
\end{equation*}Then we have
\begin{lem}[Behavior of $\phi$]\label{phiinfinitylem}Assume \eqref{W}. Let $\phi$ be the solution of
\eqref{phi}, then there exists a  constant $K_1 >0$ such that 
\begin{equation}\label{phiinfinity}\left|\phi(x)-H(x)+\frac{1}{2s\al 
}\frac{x}{|x|^{1+2s}}\right|\leq \frac{K_1}{|x|^{1+2s}},\quad\text{for }|x|\geq 1,\quad\text{if }s\ge\frac{1}{2},
\end{equation}
\begin{equation}\label{phiinfinitys<1/2}\left|\phi(x)-H(x)\right|\leq \frac{K_1}{|x|^{2s}},\quad\text{for }|x|\geq 1,\quad\text{if }s<\frac{1}{2},
\end{equation}and for any $x\in\R$, $s\in(0,1)$,
\begin{equation}\label{phi'infinity}0<
\phi'(x)\leq\frac{K_1}{1+|x|^{1+2s}},\end{equation}
\begin{equation}\label{phi''infinity}
|\phi''(x)|\leq\frac{K_1}{1+|x|^{1+2s}},\end{equation}
\begin{equation}\label{phi'''infinity}
|\phi'''(x)|\leq\frac{K_1}{1+|x|^{1+2s}}.\end{equation}
\end{lem}
\begin{proof}
Estimate \eqref{phiinfinity} is proved in \cite{dpv}, while estimates \eqref{phiinfinitys<1/2} and \eqref{phi'infinity} are proved  in \cite{psv}. 

Since the proof of \eqref{phi''infinity} and \eqref{phi'''infinity} is an adaptation of the one given in \cite{mp2} for the same estimates in the case $s=\frac{1}{2}$, we only sketch it.

To get  \eqref{phi''infinity}, as in the proof of Lemma 3.1 in \cite{mp2} one looks to the equations satisfied by $\overline{\phi}:=\phi''-C\phi_a'(x)$, where $\phi_a'(x):=\phi'\left(\frac{x}{a}\right)$, $a>0$:
$$\I[\overline{\phi}]-W''(\phi)\overline{\phi}=C\phi_a'\left(W''(\phi)-\frac{1}{a^{2s}}W''(\phi_a)\right)+W'''(\phi)(\phi')^2.$$
For $a$ and $R_1$ large enough, we can prove that in $\R\setminus[-R_1,R_1]$ we have
$$\I[\overline{\phi}]-W''(\phi)\overline{\phi}\geq 0\quad\text{and}\quad  W''(\phi)>0.$$
Choosing $C$ so large that $\overline{\phi}\leq 0$ on $[-R_1,R_1]$, the comparison principle implies $\overline{\phi}\leq0$ in $\R$, therefore $\phi''\leq C\phi_a'(x)$ in $\R$.
Similarly one can prove that $\phi''\geq-C\phi_a'(x)$ in $\R$, and using \eqref{phi'infinity},  \eqref{phi''infinity} follows. 

In the same way, comparing $\phi'''$ with $C\phi_a'(x)$, we get estimate  \eqref{phi'''infinity}.
\end{proof}

\begin{lem}[Behavior of $\psi$]\label{psiinfinitylem}Assume \eqref{W}. Let $\psi$ be the solution of
\eqref{psi}, then for any $L\in\R$  there exist  $K_2$
and $K_3>0$, depending on $L$ such that
\begin{equation}\label{psiinfinity}\left|\psi(x)-K_2\frac{x}{
|x|^{1+2s}}\right|\leq\frac{K_3}{|x|^{1+2s}},\quad\text{for }|x|\geq 1, \text{ if }s\geq\frac{1}{2}
\end{equation}and for any $s\in(0,1)$ and  $x\in\R$
\begin{equation}\label{psi'infinity}
|\psi'(x)|\leq \frac{K_3}{1+|x|^{1+2s}},
\end{equation}
\begin{equation}\label{psi''infinity}
|\psi''(x)|\leq \frac{K_3}{1+|x|^{1+2s}}.
\end{equation}
\end{lem}
\begin{proof}
We  follow the proof of Lemma 3.2 in \cite{mp2}.   Let us start with the proof of \eqref{psiinfinity}. Since we want to point out where we use $s\geq\frac{1}{2}$, we give it in the details. For $a>0$ we denote
$\phi_a(x):=\phi\left(\frac{x}{a}\right)$, which is solution of
$$\I[\phi_a]=\frac{1}{a^{2s}}W'(\phi_a)\quad\text{in }\R.$$ In what follows, we denote $\widetilde{\phi}(x)=\phi(x)-H(x).$
Let $a$ and $b$ be positive numbers, then making a Taylor
expansion of the derivatives of $W$ (remind $W'(0)=0$), we get
\begin{equation*}\begin{split}
\I[\psi-(\phi_a-\phi_b)]&=W''(\phi)\psi+\frac{L}{\al}(W''(\phi)-W''(0))+c\phi'+\left(\frac{1}{b^{2s}}W'(\phi_b)-\frac{1}{a^{2s}}W'(\phi_a)\right)\\&=
W''(\phi)(\psi-(\phi_a-\phi_b))+W''(\widetilde{\phi})(\phi_a-\phi_b)+\frac{L}{\al}(W''(\widetilde{\phi})-W''(0))\\&+c\phi'
+\left(\frac{1}{b^{2s}}W'(\widetilde{\phi}_b)-\frac{1}{a^{2s}}W'(\widetilde{\phi}_a)\right)\\&
=W''(\phi)(\psi-(\phi_a-\phi_b))+W''(0)(\phi_a-\phi_b)+\frac{L}{\al}W'''(0)\widetilde{\phi}
+c\phi'\\&+W''(0)\left(\frac{1}{b^{2s}}\widetilde{\phi}_b-\frac{1}{a^{2s}}\widetilde{\phi}_a\right)
+(\phi_a-\phi_b)O(\widetilde{\phi})+O(\widetilde{\phi})^2+O(\widetilde{\phi}_a)^2+
O(\widetilde{\phi}_b)^2.
\end{split}\end{equation*}
Then the function
$\overline{\psi}=\psi-(\phi_a-\phi_b)$ satisfies
\begin{equation*}\begin{split}\I[\overline{\psi}]-W''(\phi)\overline{\psi}&=\al(\phi_a-\phi_b)+\frac{L}{\al}W'''(0)\widetilde{\phi}+c\phi'
+\al\left(\frac{1}{b^{2s}}\widetilde{\phi}_b-\frac{1}{a^{2s}}\widetilde{\phi}_a\right)\\&+(\phi_a-\phi_b)O(\widetilde{\phi})+O(\widetilde{\phi})^2
+O(\widetilde{\phi}_a)^2+O(\widetilde{\phi}_b)^2.\end{split}\end{equation*}
We want to estimate the right-hand side of the last equality. By
Lemma \ref{phiinfinitylem}, for $|x|\geq\max\{1,|a|,|b|\}$ we have
\beqs\begin{split}\al(\phi_a-\phi_b)+\frac{L}{\al}W'''(0)\widetilde{\phi}&\geq- \frac{x}{2s
|x|^{1+2s}}\left[(a^{2s}-b^{2s})+\frac{L}{\al^2}W'''(0)\right]\\&-\frac{K_1\al}{
|x|^{1+2s}}\left(a^{1+2s}+b^{1+2s}+\frac{|L|}{\al^2}|W'''(0)|\right).\end{split}\eeqs Choose
$a,b>0$ such that $(a^{2s}-b^{2s})+\frac{L}{\al^2}W'''(0)=0$, then
$$\al(\phi_a-\phi_b)+\frac{L}{\al}W'''(0)\widetilde{\phi}\geq
-\frac{C}{|x|^{1+2s}},$$ for $|x|\geq\max\{1,|a|,|b|\}$. Here and in what
follows, as usual $C$ denotes various positive constants. From
Lemma \ref{phiinfinitylem} we also derive that
$$\al\left(\frac{1}{b^{2s}}\widetilde{\phi}_b-\frac{1}{a^{2s}}\widetilde{\phi}_a\right)\geq
-\frac{C}{|x|^{1+2s}},$$ $${\mbox{and }} \qquad c\phi'\geq -\frac{C}{1+|x|^{1+2s}}.$$  
Moreover, since $s\geq\frac{1}{2}$, we have 
$$(\phi_a-\phi_b)O(\widetilde{\phi})+O(\widetilde{\phi})^2+O(\widetilde{\phi}_a)^2+O(\widetilde{\phi}_b)^2\geq
-\frac{C}{1+|x|^{4s}}\geq -\frac{C}{1+|x|^{1+2s}},$$ for $|x|\geq\max\{1,|a|,|b|\}$. Then we
conclude that there exists $R_1>0$ such that for $|x|\geq R_1$ we have
$$\I[\overline{\psi}]-W''(\phi)\overline{\psi}\geq
-\frac{C}{1+|x|^{1+2s}}.$$ 

\noindent Now, let us consider the function
$\phi'_d(x)=\phi'\left(\frac{x}{d}\right)$, $d>0$, which is
solution of
$$\I[\phi_d']=\frac{1}{d^{2s}}W''(\phi_d)\phi_d'\quad\text{in }\R,$$ and denote
$$\overline{\overline{\psi}}=\overline{\psi}-\widetilde{C}\phi_d',$$with $\widetilde{C}>0$. Then, for $|x|\geq R_1$ we have
\beqs\begin{split}\I[\overline{\overline{\psi}}]&\geq
W''(\phi)\overline{\psi}-\frac{\widetilde{C}}{d^{2s}}W''(\phi_d)\phi_d'-\frac{C}{1+|x|^{1+2s}}\\&=W''(\phi)\overline{\overline{\psi}}
+\widetilde{C}\phi_d'\left(W''(\phi)-\frac{1}{d^{2s}}W''(\phi_d)\right)-\frac{C}{1+|x|^{1+2s}}.\end{split}\eeqs
Let us choose $d>0$ and $R_2>R_1$ such that
\begin{equation*}\left\{%
\begin{array}{ll}W''(\phi)-\frac{1}{d^{2s}}W''(\phi_d)>\frac{1}{2}W''(0)>0&\text{in }\R\setminus[-R_2,R_2];\\
W''(\phi)>0&\text{on }\R\setminus[-R_2,R_2],\\
\end{array}%
\right.\end{equation*} then from \eqref{phi'infinity},  for $\widetilde{C}$ large enough we get
\begin{equation*}\I[\overline{\overline{\psi}}]-W''(\phi)\overline{\overline{\psi}}\geq
0\quad\text{on }\R\setminus[-R_2,R_2].\end{equation*} 
Choosing $\widetilde{C}$ such that moreover
$$\overline{\overline{\psi}}<0\quad \text{on
}[-R_2,R_2],$$
we can ensure that $\overline{\overline{\psi}}\leq 0$ on $\R$.
Indeed, assume by contradiction that there exists $x_0\in
\R\setminus[-R_2,R_2]$ such that
$$\overline{\overline{\psi}}(x_0)=\sup_{\R}\overline{\overline{\psi}}>0.$$Then
$$\left\{%
\begin{array}{ll}
    \I[\overline{\overline{\psi}},x_0]\leq 0; \\
    \I[\overline{\overline{\psi}},x_0]-W''(\phi(x_0))\overline{\overline{\psi}}(x_0)\geq
0; \\
W''(\phi(x_0))>0,\\
\end{array}%
\right.
$$from which $$\overline{\overline{\psi}}(x_0)\leq 0,$$ a
contradiction. Therefore, $\overline{\overline{\psi}}\leq 0$ on $\R$ which implies, with together \eqref{phiinfinity} and \eqref{phi'infinity},   
$$\psi\leq \frac{K_2x}{|x|^{1+2s}}+\frac{K_3}{|x|^{1+2s}}\quad\text{for }|x|\geq1.$$
Looking at the function $\psi
-(\phi_a-\phi_b)+\widetilde{C}\phi_d'$, we conclude similarly that
$$\psi\geq \frac{K_2x}{|x|^{1+2s}}-\frac{K_3}{|x|^{1+2s}}\quad\text{for }|x|\geq1,$$ and \eqref{psiinfinity} is proved.

Now let us turn to \eqref{psi'infinity}. By deriving the first
equation in \eqref{psi}, we see that the function $\psi'$  which
is bounded and of class $C^{2,\beta}$, is a solution of
$$\I[\psi']=W''(\phi)\psi'+W'''(\phi)\phi'\psi+\frac{L}{\al}W'''(\phi)\phi'+c\phi''\quad\text{in }\R.$$
 Then the function $\overline{\psi}'=\psi'-C\phi_a'$, satisfies
\begin{equation*}\begin{split}\I[\overline{\psi}']-W''(\phi)\overline{\psi}'&=C\phi'_a\left(W''(\phi)-\frac{1}{a^{2s}}W''(\phi_a)\right)+W'''(\phi)\phi'\psi
+\frac{L}{\al}W'''(\phi)\phi'+c\phi''\\&=C\phi'_a\left(W''(\phi)-\frac{1}{a^{2s}}W''(\phi_a)\right)+O\left(\frac{1}{1+|x|^{1+2s}}\right),
\end{split}\end{equation*}by \eqref{phi'infinity} and  \eqref{phi''infinity}
and as before we deduce that for $C$ and $a$ large enough
$\overline{\psi}'\leq 0$ on $\R$, which implies that $\psi'\leq
\frac{K_3}{1+|x|^{1+2s}}$. The inequality $\psi'\geq -\frac{K_3}{1+|x|^{1+2s}}$
is obtained similarly by proving that
$\overline{\psi}'+C\phi'_a\geq 0$ on $\R$.

Similarly, estimate \eqref{psi''infinity} is obtained by comparing  $\psi''$ with $C\phi'_a$ for some large $a$ and $C$ and using 
\eqref{phi'infinity}, \eqref{phi''infinity} and \eqref{phi'''infinity}. 
\end{proof}

\subsection{Proof of Proposition \ref{hproperties}}\label{s-s4}
{\mbox{ }} \bigskip

For simplicity of notation we denote (for the rest of the paper)
$$x_i=\frac{x-i}{\delta|p_0|},\quad\widetilde{\phi}(z)=\phi(z)-H(z).$$
 Then we have the following six claims
(whose proofs are  postponed to the end of the section).\\
\bigskip

\noindent {\bf Claim 1:} {\em Let $x=i_0+\gamma$, with $i_0\in\Z$
and $\gamma\in\left(-\frac{1}{2},\frac{1}{2}\right]$, then there exist numbers $\theta_i\in(-1,1)$ such that 
\begin{equation*}\SUM_{i=-n\atop i\neq
i_0}^n\frac{x-i}{|x-i|^{1+2s}}\rightarrow -4s\gamma\SUM_{i=1}^{
+\infty}\frac{(i-\theta_i \gamma)^{2s-1}}{(i+\gamma)^{2s}(i-\gamma)^{2s}}\quad\text{as
}n\rightarrow+\infty,\end{equation*}
\begin{equation*}
\SUM_{i=-n\atop }^{i_0-1}\frac{1}{|x-i|^{1+2s}}\rightarrow
\SUM_{i=1}^{+\infty}\frac{1}{(i+\gamma)^{1+2s}}\quad\text{as
}n\rightarrow+\infty,\end{equation*}
\begin{equation*}\SUM_{i=i_0+1\atop }^{n}\frac{1}{|x-i|^{1+2s}}\rightarrow
\SUM_{i=1}^{+\infty}\frac{1}{(i-\gamma)^{1+2s}}\quad\text{as
}n\rightarrow+\infty.\end{equation*}}\\
We remark that the three series on the right hand side above   converge uniformly for $\gamma\in\left(-\frac{1}{2},\frac{1}{2}\right]$
and $\theta_i\in(-1,1)$ since behave like the series  $\SUM_{i=1}^{+\infty}\frac{1}{i^{1+2s}}$.
\bigskip

\noindent {\bf Claim 2:} {\em Assume $s<\frac{1}{2}$. Let $x=i_0+\gamma$, with $i_0\in\Z$
and $\gamma\in\left(-\frac{1}{2},\frac{1}{2}\right]$, then

\beq\label{phisums<1/2estim}\LIM_{n\rightarrow\infty} \left|\SUM_{i=-n\atop
i\neq i_0}^n[\widetilde{\phi}(x_i)]^{2k-1}\right|\leq Ck\delta^{2s(2k-1)}|\gamma|
\eeq and 

\beq\label{Itausum}\SUM_{i=-n\atop
i\neq i_0,i_0\pm1}^n|\I[\tau,x_i]|\leq C\delta^{2s}.
\eeq\\}

\noindent {\bf Claim 3:} {\em For any $x\in\R$ the sequence
$\{s_{\delta,n}^L(x)\}_n$ converges as $n\rightarrow+\infty$.}\\

\noindent {\bf Claim 4:} {\em The sequence
$\{(s_{\delta,n}^L)'\}_n$ converges on $\R$ as
$n\rightarrow+\infty$, uniformly on compact sets.} \\

\noindent {\bf Claim 5:} {\em The sequence
$\{(s_{\delta,n}^L)''\}_n$ converges
on $\R$ as $n\rightarrow+\infty$, uniformly on compact sets.}\\

\noindent {\bf Claim 6:} {\em For any $x\in\R$ the sequence
$\SUM_{i=-n}^n\I[s_{\delta,n}^L,x_i]$ 
converges as $n\rightarrow+\infty$.}\\

With these claims, we are in the position of completing
the proof of Proposition \ref{hproperties}, by arguing as follows.\medskip

\noindent {\bf Proof of ii)}\\
When  $s\geq\frac{1}{2}$, (ii) is a consequence of \eqref{sdeltan-xbddclaim3} in the proof of Claim 3.

Next, let us assume $s<\frac{1}{2}$.
Let   $x=i_0+\gamma$ with $i_0\in\Z$ and
$\gamma\in\left(-\frac{1}{2},\frac{1}{2}\right]$. 
For $n>|i_0|$, we have
\beqs\begin{split}\SUM _{i=-n}^{i=n}\phi(x_i)-n-x&=\SUM _{i=-n}^{i=n}\phi(x_i)-n-i_0-\gamma\\&
=\SUM_{i=-n}^{i_0-1}(\phi(x_i)-1)+\phi(x_{i_0})+\SUM_{i=i_0+1}^{n}\phi(x_i)-\gamma\\&
= \SUM_{i=-n\atop i\neq i_0}^{i=n}\widetilde{\phi}(x_i)+\phi(x_{i_0})-\gamma.
\end{split}\end{equation*} 
Then by \eqref{phisums<1/2estim} with $k=1$
\beq\label{phi-xbdds<12}\left|\SUM _{i=-n}^{i=n}\phi(x_i)-n-x\right|\leq C,\eeq with $C$ independent of $x$. 
Finally, for $i\neq i_0-1,i_0,i_0+1$ and $R=\frac{1}{2\delta|p_0|}$
$$|x_i|=\frac{|i_0+\gamma-i|}{\delta|p_0|}\geq \frac{3}{2\delta|p_0|}>2R,$$ therefore $\tau(x_i)=0$. This implies that $\SUM_{i=-n}^{n}\psi(x_i)\tau(x_i)$ is actually the  sum of only three terms and therefore 
\beq\label{psibdds<12}\left|\SUM _{i=-n}^{i=n}\psi(x_i)\tau(x_i)\right|\leq 3\|\psi\|_\infty.\eeq
Estimates \eqref{phi-xbdds<12} and \eqref{psibdds<12} imply (ii).
\medskip

\noindent {\bf Proof of i)}\\
The function
$h_\delta^L(x)=\LIM_{n\rightarrow+\infty}s_{\delta,n}^L(x)$ is
well defined for any $x\in\R$ by Claim 3. Moreover, by Claims 4 and
5 and classical analysis results, it is of class $C^2$ on $\R$
with
$$(h_\delta^{L})'(x)=\LIM_{n\rightarrow+\infty}(s_{\delta,n}^L)'(x),$$
$$(h_\delta^{L})''(x)=\LIM_{n\rightarrow+\infty}(s_{\delta,n}^L)''(x),$$
and the
convergence of $\{s_{\delta,n}^L\}_n$, $\{(s_{\delta,n}^L)'\}_n$
and $\{(s_{\delta,n}^L)''\}_n$ is uniform on compact sets.

Finally, as in \cite{mp2} (see  Section 4), we have
 for any $x\in\R$
\begin{equation}\label{NLh}\I[h_\delta^L,x]=\LIM_{n\rightarrow+\infty}\I[s_{\delta,n}^L,x].\end{equation}

To conclude the proof of Proposition \ref{hproperties}, we only have to prove \eqref{nl=odelta}, which is a consequence of the estimates above and the following lemma.


\begin{lem}{\bf (First asymptotics)}\label{lem::1011}
We have
$$\LIM_{n\rightarrow+\infty}NL_L^{\overline{\lam}_\delta^L}[s_{\delta,n}^L](x)=o(\delta^{2s})\quad\text{as }\delta\rightarrow0$$ 
where $\LIM_{\delta\rightarrow0}\frac{o(\delta^{2s})}{\delta^{2s}}=0$, uniformly for $x\in\R$. 
\end{lem}

Now we can conclude the proof of (i). Indeed, by Claim 3, Claim 4
and \eqref{NLh}, for any $x\in\R$
$$NL_L^{\overline{\lam}_\delta^L}[h_\delta^L](x)=\LIM_{n\rightarrow+\infty}NL_L^{\overline{\lam}_\delta^L}[s_{\delta,n}^L](x),$$
and Lemma \ref{lem::1011} implies that
$$NL_L^{\overline{\lam}_\delta^L}[h_\delta^L](x)=o(\delta^{2s}),\quad\text{as
}\delta\rightarrow0,$$where
$\LIM_{\delta\rightarrow0}\frac{o(\delta^{2s})}{\delta^{2s}}=0$, uniformly
for $x\in\R$. 
\bigskip

\noindent {\bf Proof of Lemma \ref{lem::1011}.}\\

Let us first assume $s\geq\frac{1}{2}$.\\

\noindent {\bf Step 1: First computation}\\
Fix $x\in\R$, let $i_0\in\Z$ and
$\gamma\in\left(-\frac{1}{2},\frac{1}{2}\right]$ be such that
$x=i_0+\gamma$, let $\frac{1}{\delta|p_0|}\geq2$ and $n>|i_0|$.
Then we have 

\begin{equation*}\begin{split}
A:= \quad &NL_L^{\overline{\lam}_\delta^L}[s_{\delta,n}^L](x)\\
\\
=\quad &\frac{\overline{\lam}_\delta^L}{\delta|p_0|}\SUM_{i=-n}^n\left[\phi'(x_i)+\delta^{2s}
\psi'(x_i)\right]-\SUM_{i=-n}^n\left[\I[\phi,x_i]+\delta^{2s}\I[\psi,x_i]\right]\\&+W'\left(\frac{L\delta^{2s}}{\al}
+\SUM_{i=-n}^n\left[\phi(x_i)+\delta^{2s}
\psi(x_i)\right]\right)-\delta^{2s} L
\end{split}\end{equation*}
where we have used the definitions and the periodicity of $W$.
Using the equation (\ref{phi}) satisfied by $\phi$, we can rewrite it as

\begin{equation*}\begin{split}
A=\quad &\frac{\overline{\lam}_\delta^L}{\delta|p_0|}\left\{\phi'(x_{i_0})+\delta^{2s}
\psi'(x_{i_0})+\SUM_{i=-n\atop i\neq i_0}^n\left[\phi'(x_i)+\delta^{2s}
\psi'(x_i)\right]\right\} -\SUM_{i=-n\atop i\neq
i_0}^nW'(\widetilde{\phi}(x_i))\\
&-\delta^{2s}\SUM_{i=-n\atop
i\neq i_0}^n\I[\psi,x_{i}]-\delta^{2s}\I[\psi,x_{i_0}]\\&+W'\left(\frac{L\delta^{2s}}{\al}
+\SUM_{i=-n}^n\left[\widetilde{\phi}(x_i)+\delta^{2s}
\psi(x_i)\right]\right)-W'(\widetilde{\phi}(x_{i_0}))-\delta^{2s}
L.
\end{split}\end{equation*}
Using the Taylor expansion of $W'$ (remind that $W'(0)=0$) and the definition of $\overline{\lam}_\delta^L$, we get  
\begin{equation*}\begin{split}
A=\quad &\delta^{2s} c_0L\left\{\phi'(x_{i_0})+\delta^{2s}
\psi'(x_{i_0})+\SUM_{i=-n\atop i\neq i_0}^n\left[\phi'(x_i)+\delta^{2s}
\psi'(x_i)\right]\right\}\\&-W''(0)\SUM_{i=-n\atop i\neq
i_0}^n\widetilde{\phi}(x_i)-\delta^{2s}\SUM_{i=-n\atop i\neq
i_0}^n\I[\psi,x_{i}]
-\delta^{2s}\I[\psi,x_{i_0}]\\&+W''(\phi(x_{i_0}))\left(\frac{L\delta^{2s}}{\al}
+\delta^{2s} \psi(x_{i_0})+\SUM_{i=-n\atop i\neq
i_0}^n\left[\widetilde{\phi}(x_i)+\delta ^{2s}\psi(x_i)\right]\right)-\delta^{2s} L +E
\end{split}\end{equation*}
with the error term
$$E=E_1+E_2,$$ where 
$$E_1=-\SUM_{i=-n\atop i\neq i_0}^nW'(\widetilde{\phi}(x_i))+W''(0)\SUM_{i=-n\atop i\neq
i_0}^n\widetilde{\phi}(x_i)$$ 
and
 $$E_2=O\left(\frac{L\delta^{2s}}{\al}
+\delta^{2s} \psi(x_{i_0})+\SUM_{i=-n\atop i\neq
i_0}^n\left[\widetilde{\phi}(x_i)+\delta^{2s}
\psi(x_i)\right]\right)^2.$$
Simply reorganizing the terms, we get with $c=c_0 L$:

\begin{equation*}\begin{split}
A=\quad &\delta^{2s}  c\left\{\delta^{2s} 
\psi'(x_{i_0})+\SUM_{i=-n\atop i\neq i_0}^n\left[\phi'(x_i)+\delta^{2s} 
\psi'(x_i)\right]\right\}-W''(0)\SUM_{i=-n\atop i\neq
i_0}^n\widetilde{\phi}(x_i)-\delta^{2s} \SUM_{i=-n\atop i\neq
i_0}^n\I[\psi,x_{i}]\\
&+W''(\phi(x_{i_0}))\left(\SUM_{i=-n\atop
i\neq i_0}^n\left[\widetilde{\phi}(x_i)+\delta^{2s} 
\psi(x_i)\right]\right)\\
&+\delta^{2s} \Big(-\I[\psi,x_{i_0}]+W''(\phi(x_{i_0}))\psi(x_{i_0})+\frac{L}{\al}W''(\phi(x_{i_0}))-L+c\phi'(x_{i_0})\Big)+E.
\end{split}\end{equation*}

Using equation (\ref{psi}) satisfied by $\psi$, we get 

\begin{equation*}\begin{split}
A=\quad &\delta^{2s} c\left\{\delta^{2s} \psi'(x_{i_0})+\SUM_{i=-n\atop i\neq
i_0}^n\left[\phi'(x_i)+\delta^{2s}
\psi'(x_i)\right]\right\}+(W''(\phi(x_{i_0}))-W''(0))\SUM_{i=-n\atop
i\neq i_0}^n\widetilde{\phi}(x_i)\\
&-\delta^{2s}\SUM_{i=-n\atop i\neq
i_0}^n\I[\psi,x_{i}]+W''(\phi(x_{i_0}))\delta^{2s}\SUM_{i=-n\atop i\neq
i_0}^n \psi(x_i)+E.
\end{split}\end{equation*}

Let us bound the second term  of the last equality, uniformly in
$x$.\\

\noindent {\bf Step 2: Bound on $\SUM_{i=-n\atop i\neq i_0}^n\left[\phi'(x_i)+\delta^{2s}
\psi'(x_i)\right]$}\\

 From \eqref{phi'infinity} and \eqref{psi'infinity} it follows
that
\beqs\begin{split}&-\delta^{1+4s}|p_0|^{1+2s} K_3\SUM_{i=-n\atop
i\neq i_0}^n\frac{1}{|x-i|^{1+2s}}\leq\SUM_{i=-n\atop i\neq
i_0}^n\left[\phi'(x_i)+\delta^{2s} \psi'(x_i)\right]\\&\leq
\delta^{1+2s}|p_0|^{1+2s}(K_1+\delta^{2s} K_3)\SUM_{i=-n\atop i\neq
i_0}^n\frac{1}{|x-i|^{1+2s}},\end{split}\eeqs and then by Claim 1 we get
\begin{equation}\label{NL1}-C\delta^{1+4s}\leq\LIM_{n\rightarrow+\infty}\SUM_{i=-n\atop i\neq i_0}^n\left[\phi'(x_i)+\delta^{2s}
\psi'(x_i)\right]\leq C\delta^{1+2s}.\end{equation} Here and
henceforth, $C$ denotes various positive constants independent of
$x$.\\
 

\noindent {\bf Step 3: Bound on $(W''(\phi(x_{i_0}))-W''(0))\SUM_{i=-n\atop
i\neq i_0}^n\widetilde{\phi}(x_i)$}\\

Let us prove that
\begin{equation}\label{NL2}\LIM_{n\rightarrow+\infty}\left|(W''(\phi(x_{i_0}))-W''(0))\SUM_{i=-n\atop
i\neq i_0}^n\widetilde{\phi}(x_i)\right|\leq C\delta^{1+2s}.
\end{equation}
By \eqref{phiinfinity} we have
\begin{equation}\label{nlphi1}\begin{split}\left|\SUM_{i=-n\atop i\neq
i_0}^n\widetilde{\phi}(x_i)+\frac{\delta^{2s}|p_0|^{2s}}{2s\al}\SUM_{i=-n\atop
i\neq i_0}^n\frac{x-i}{|x-i|^{1+2s}}\right|\leq
K_1\delta^{1+2s}|p_0|^{1+2s}\SUM_{i=-n\atop i\neq i_0}^n\frac{1}{|x-i|^{1+2s}}.
\end{split}\end{equation} If $|\gamma|\geq \delta|p_0|$  then $|x_{i_0}|=\frac{|\gamma|}{ \delta|p_0|}\geq 1$ and 
again from \eqref{phiinfinity}, $$|\widetilde{\phi}(x_{i_0})
+\frac{\delta^{2s}|p_0|^{2s}}{2s\al}\frac{\gamma}{|\gamma|^{1+2s}}|\leq
K_1\frac{\delta^{1+2s}|p_0|^{1+2s}}{|\gamma|^{1+2s}}$$ which implies that
$$|W''(\widetilde{\phi}(x_{i_0}))-W''(0)|\leq
|W'''(0)\widetilde{\phi}(x_{i_0})|+O(\widetilde{\phi}(x_{i_0}))^2\leq
C\frac{\delta^{2s}}{|\gamma|^{2s}}+C\frac{\delta^{1+2s}}{|\gamma|^{1+2s}}.$$ By the
previous inequality, \eqref{nlphi1} and Claim 1 we deduce that
\begin{equation*}\begin{split}\LIM_{n\rightarrow+\infty}\left|(W''(\phi(x_{i_0}))-W''(0))\SUM_{i=-n\atop
i\neq i_0}^n\widetilde{\phi}(x_i)\right|&\leq
C\left(\frac{\delta^{2s}}{|\gamma|^{2s}}+\frac{\delta^{1+2s}}{|\gamma|^{1+2s}}\right)(\delta^{2s}|\gamma|+\delta^{1+2s})\\&\leq C\delta^{1+2s}
\end{split}\end{equation*}where $C$ is independent of $\gamma$.

Finally, if $|\gamma|<\delta|p_0|$, from \eqref{nlphi1} and Claim
1 we conclude that
\begin{equation*}\left|\LIM_{n\rightarrow+\infty}(W''(\phi(x_{i_0}))-W''(0))\SUM_{i=-n\atop
i\neq i_0}^n\widetilde{\phi}(x_i)\right|\leq
C\delta^{2s}|\gamma|+C\delta^{1+2s}\leq C\delta^{1+2s},
\end{equation*}and \eqref{NL2} is proved.\\

\noindent {\bf Step 4: Bound on $\delta^{2s}\SUM_{i=-n\atop i\neq
i_0}^n\I[\psi,x_{i}]$}\\

We compute
\begin{equation}\begin{split}\label{i1psi}\I[\psi]&=W''(\widetilde{\phi})\psi+\frac{L}{\al}(W''(\widetilde{\phi})-W''(0))+c\phi'\\&
=W''(0)\psi+\frac{L}{\al}W'''(0)\widetilde{\phi}+O(\widetilde{\phi})\psi+O(\widetilde{\phi})^2+c\phi'.\end{split}\end{equation}
Estimates \eqref{phiinfinity} and \eqref{psiinfinity}  implies that the sequences $$\SUM_{i=-n\atop i\neq
i_0}^nO(\widetilde{\phi}(x_{i}))\psi(x_{i}),\quad\SUM_{i=-n\atop i\neq
i_0}^nO(\widetilde{\phi}(x_{i}))^2$$  behave like the series $\SUM_{i=1}^{\infty}\frac{1}{i^{4s}}$, therefore they are convergent since $s\geq\frac{1}{2}$. 
Moreover, by \eqref{i1psi}, \eqref{phiinfinity}, \eqref{phi'infinity},
\eqref{psiinfinity} and Claim 1, we have
\begin{equation}\label{NL3}\left|\LIM_{n\rightarrow+\infty}\delta^{2s}\SUM_{i=-n\atop i\neq
i_0}^n\I[\psi,x_{i}]\right|\leq C(\delta^{4s}+\delta^{1+2s})\leq C\delta^{1+2s}.
\end{equation}\\

\noindent {\bf Step 5: Bound on $W''(\phi(x_{i_0}))\delta^{2s}\SUM_{i=-n\atop
i\neq i_0}^n \psi(x_i)$}\\
Similarly, from \eqref{psiinfinity} and Claim 1 we get 
\begin{equation}\label{NL4}\left|\LIM_{n\rightarrow+\infty}W''(\phi(x_{i_0}))\delta^{2s}\SUM_{i=-n\atop
i\neq i_0}^n \psi(x_i)\right|\leq C\delta^{1+2s}.
\end{equation}

\noindent {\bf Step 6: Bound on the error $E$}\\
Finally, again from \eqref{phiinfinity}, \eqref{psiinfinity}
 and Claim 1 it follows
that
\begin{equation}\label{NL5}\left|\LIM_{n\rightarrow+\infty}E_2\right|=\left|\LIM_{n\rightarrow+\infty}O\left(\frac{L\delta^{2s}}{\al} +\delta^{2s}
\psi(x_{i_0})+\SUM_{i=-n\atop i\neq
i_0}^n\left[\widetilde{\phi}(x_i)+\delta^{2s}
\psi(x_i)\right]\right)^2\right|\leq C\delta^{4s}\leq C\delta^{1+2s}.
\end{equation}
Next, let us  estimate $E_1$. From  \eqref{phiinfinity}  and using $s\geq\frac{1}{2}$, we have 
\beq\label{NL5'} |E_1|\leq\SUM_{i=-n\atop i\neq i_0}^n|W'(\widetilde{\phi}(x_i))-W''(0)\widetilde{\phi}(x_i)|=\SUM_{i=-n\atop i\neq i_0}^n|O(\widetilde{\phi}(x_i))^2|\leq C\delta^{4s}\leq C\delta^{1+2s}.\eeq

\noindent {\bf Step 7: Conclusion}\\
Therefore, from \eqref{NL1}, \eqref{NL2}, \eqref{NL3}, \eqref{NL4}, \eqref{NL5} and
\eqref{NL5'} we conclude that
$$-C\delta^{1+2s}\leq \LIM_{n\rightarrow+\infty}NL_L^{\overline{\lam}_\delta^L}[s_{\delta,n}^L]\leq C\delta^{1+2s}$$ 
with $C$ independent of $x$ and
Lemma \ref{lem::1011} for $s\geq\frac{1}{2}$ is
proved.\\


Now, let us turn to the case $s<\frac{1}{2}$.\\

\noindent {\bf Step 1': First computation}\\

Making computations like in Step 1, we get
\begin{equation*}\begin{split}
A&:=NL_L^{\overline{\lam}_\delta^L}[s_{\delta,n}^L](x)\\& 
=\delta^{2s} c\left\{\delta^{2s} (\psi\tau)'(x_{i_0})+\SUM_{i=-n\atop i\neq
i_0}^n\left[\phi'(x_i)+\delta^{2s}
(\psi\tau)'(x_i)\right]\right\}\\&+(W''(\phi(x_{i_0}))-W''(0))\SUM_{i=-n\atop
i\neq i_0}^n\widetilde{\phi}(x_i)\\&-\delta^{2s}\SUM_{i=-n\atop i\neq
i_0}^n\I[\psi\tau,x_{i}]+W''(\phi(x_{i_0}))\delta^{2s}\SUM_{i=-n\atop i\neq
i_0}^n (\psi\tau)(x_i)
\\
&+\delta^{2s} \Big(-\I[\psi\tau,x_{i_0}]+W''(\phi(x_{i_0}))(\psi\tau)(x_{i_0})+\frac{L}{\al}W''(\phi(x_{i_0}))-L+c\phi'(x_{i_0})\Big)+
E,
\end{split}\end{equation*} where again $E=E_1+E_2$ with
$E_1$  the error term coming from in the approximation of  $\SUM_{i=-n\atop i\neq i_0}^nW'(\widetilde{\phi}(x_i))$ with $W''(0)\SUM_{i=-n\atop i\neq
i_0}^n\widetilde{\phi}(x_i)$,
and
 $$E_2=O\left(\frac{L\delta^{2s}}{\al}
+\delta^{2s} \psi(x_{i_0})+\SUM_{i=-n\atop i\neq
i_0}^n\left[\widetilde{\phi}(x_i)+\delta^{2s}
\psi(x_i)\tau(x_i)\right]\right)^2.$$
 To control the term $\I[\psi\tau,x_{i}]$, we use the following formula 
 which can be found for instance in \cite{bpsv} page 7:
\beq\label{Iproductformula} \I[\psi\tau,x_i]=\tau(x_i) \I[\psi,x_i]+\psi(x_i) \I[\tau,x_i]-B(\psi,\tau)(x_i),\eeq

where
$$B(\psi,\tau)(x_i)=C(s)\int_\R\frac{(\psi(y)-\psi(x_i))(\tau(y)-\tau(x_i))}{|x_i-y|^{1+2s}}dy.$$
Therefore the quantity $A$ can be rewritten in the following way:

\begin{equation*}\begin{split}
A=\quad &\delta^{2s} c\left\{\delta^{2s} (\psi\tau)'(x_{i_0})+\SUM_{i=-n\atop i\neq
i_0}^n\left[\phi'(x_i)+\delta^{2s}
(\psi\tau)'(x_i)\right]\right\}\\&+(W''(\phi(x_{i_0}))-W''(0))\SUM_{i=-n\atop
i\neq i_0}^n\widetilde{\phi}(x_i)\\&-\delta^{2s}\SUM_{i=-n\atop i\neq
i_0}^n\I[\psi\tau,x_{i}]+W''(\phi(x_{i_0}))\delta^{2s}\SUM_{i=-n\atop i\neq
i_0}^n (\psi\tau)(x_i)
\\
&+\delta^{2s} \Big(-\tau(x_{i_0})\I[\psi,x_{i_0}]+W''(\phi(x_{i_0}))(\psi\tau)(x_{i_0})+\frac{L}{\al}W''(\phi(x_{i_0}))-L+c\phi'(x_{i_0})\Big)\\&
+\delta^{2s}B(\psi,\tau)(x_{i_0})-\delta^{2s} \psi(x_{i_0})\I[\tau,x_{i_0}]+
E.
\end{split}\end{equation*}
Now, we
remark that $$|x_{i_0}|=\frac{|\gamma|}{\delta |p_0|}\leq\frac{1}{2\delta |p_0|}=R,$$ then by \eqref{tauproper} $\tau(x_{i_0})=1$. Therefore, using the equation satisfied by $\psi$ 
\eqref{psi}, we get 

$$-\tau(x_{i_0})\I[\psi,x_{i_0}]+W''(\phi(x_{i_0}))(\psi\tau)(x_{i_0})+\frac{L}{\al}W''(\phi(x_{i_0}))-L+c\phi'(x_{i_0})=0$$ and consequently 
\begin{equation*}\begin{split}
A=\quad &\delta^{2s} c\left\{\delta^{2s} (\psi\tau)'(x_{i_0})+\SUM_{i=-n\atop i\neq
i_0}^n\left[\phi'(x_i)+\delta^{2s}
(\psi\tau)'(x_i)\right]\right\}\\&+(W''(\phi(x_{i_0}))-W''(0))\SUM_{i=-n\atop
i\neq i_0}^n\widetilde{\phi}(x_i)\\&-\delta^{2s}\SUM_{i=-n\atop i\neq
i_0}^n\I[\psi\tau,x_{i}]+W''(\phi(x_{i_0}))\delta^{2s}\SUM_{i=-n\atop i\neq
i_0}^n (\psi\tau)(x_i)
\\
&+\delta^{2s}B(\psi,\tau)(x_{i_0})-\delta^{2s} \psi(x_{i_0})\I[\tau,x_{i_0}]+
E.
\end{split}\end{equation*}

Let us proceed to the estimate of $A$.\\

\noindent {\bf Step 2': Bound on $\SUM_{i=-n\atop i\neq i_0}^n\left[\phi'(x_i)+\delta^{2s}
(\psi\tau)'(x_i)\right]$}\\

As in Step 2, using \eqref{phi'infinity} and Claim 1, we get 
\begin{equation}\label{NL1s<1/2}0\leq\LIM_{n\rightarrow+\infty}\SUM_{i=-n\atop i\neq i_0}^n\phi'(x_i)\leq C\delta^{1+2s}.\end{equation} 

Next, for $i\neq i_0-1,i_0,i_0+1$, and $R=\frac{1}{2\delta|p_0|}$
$$|x_i|=\frac{|i_0+\gamma-i|}{\delta|p_0|}\geq \frac{3}{2\delta|p_0|}>2R,$$ therefore $\tau(x_i)=\tau'(x_i)=0$. Then, using \eqref{psi'infinity} and the fact that $\LIM_{x\rightarrow \pm \infty}\psi(x)=0$, we get
\beq\label{NL1psis<1/2} \begin{split}\delta^{2s}\SUM_{i=-n\atop i\neq i_0}^n (\psi\tau)'(x_i)&=\delta^{2s}(\psi\tau)'(x_{i_0-1})+\delta^{2s}(\psi\tau)'(x_{i_0+1})\\&
=\delta^{2s}(\psi\tau)'\left(\frac{-1+\gamma}{\delta|p_0|}\right)+\delta^{2s}(\psi\tau)'\left(\frac{1+\gamma}{\delta|p_0|}\right)
\\&=o(\delta^{2s}).\end{split}\eeq\\

\noindent {\bf Step 3': Bound on $(W''(\phi(x_{i_0}))-W''(0))\SUM_{i=-n\atop
i\neq i_0}^n\widetilde{\phi}(x_i)$}\\

From \eqref{phisums<1/2estim} with $k=1$ we know that 

$$\LIM_{n\rightarrow+\infty}\left|\SUM_{i=-n\atop
i\neq i_0}^n\widetilde{\phi}(x_i)\right|\leq C\delta^{2s}|\gamma|.$$
As in Step 3  if $|\gamma|\geq \delta|p_0|$, then \eqref{phiinfinitys<1/2}  implies
$$|\widetilde{\phi}(x_{i_0})|\leq C\frac{\delta^{2s}}{|\gamma|^{2s}},$$ and so, using that  $W'''(0)=0$
$$|W''(\widetilde{\phi}(x_{i_0}))-W''(0)|\leq C\frac{\delta^{4s}}{|\gamma|^{4s}}.$$
Then we have
$$\LIM_{n\rightarrow+\infty}\left|(W''(\phi(x_{i_0}))-W''(0))\SUM_{i=-n\atop
i\neq i_0}^n\widetilde{\phi}(x_i)\right|\leq C\frac{\delta^{4s}}{|\gamma|^{4s}}\delta^{2s}|\gamma|\leq C\delta^{4s}.$$
Finally, if $|\gamma|<\delta|p_0|$, then 
$$\LIM_{n\rightarrow\infty} \left|\SUM_{i=-n\atop
i\neq i_0}^n\widetilde{\phi}(x_i)\right|\leq C\delta^{2s}|\gamma|\leq C\delta^{1+2s}.$$ We conclude that 
\beq\label{NL2s<1/2} \LIM_{n\rightarrow+\infty}\left|(W''(\phi(x_{i_0}))-W''(0))\SUM_{i=-n\atop
i\neq i_0}^n\widetilde{\phi}(x_i)\right|\leq C\delta^{4s}.\eeq

\noindent {\bf Step 4': Bound on $\delta^{2s}\SUM_{i=-n\atop i\neq
i_0}^n\I[\psi\tau,x_{i}]$}\\

Using formula \eqref{Iproductformula}, we see that 
\beq\label{ipsitausplitproduct}\delta^{2s}\SUM_{i=-n\atop i\neq
i_0}^n\I[\psi\tau,x_{i}]=\delta^{2s}\SUM_{i=-n\atop i\neq
i_0}^n\left\{\tau(x_{i})\I[\psi,x_{i}]+\psi(x_{i})\I[\tau,x_{i}]-B(\psi,\tau)(x_{i})\right\}.\eeq
As we have already pointed out in Step 2', for $i\neq i_0-1,i_0,i_0+1$, $\tau(x_i)=0$, therefore 
$$\delta^{2s}\SUM_{i=-n\atop i\neq
i_0}^n\tau(x_{i})\I[\psi,x_{i}]=\delta^{2s}\tau(x_{i_0-1})\I[\psi,x_{i_0-1}]+\delta^{2s}\tau(x_{i_0+1})\I[\psi,x_{i_0+1}].$$
We point out that $$x_{i_0-1}=\frac{-1+\gamma}{\delta|p_0|}\rightarrow -\infty\quad\text{as }\delta\rightarrow 0$$ and 
$$x_{i_0+1}=\frac{1+\gamma}{\delta|p_0|}\rightarrow+\infty\quad\text{as }\delta\rightarrow 0.$$ 
Then from the equation \eqref{psi}, estimates \eqref{phiinfinity}, \eqref{phi'infinity} and $\LIM_{x\rightarrow\pm\infty}\psi(x)=0$, we deduce that $\I[\psi,x_{i_0-1}]$ and $\I[\psi,x_{i_0+1}]$ are $o(1)$ as $\delta\rightarrow 0$, this implies that 
\beq\label{tauxiIpsi}\delta^{2s}\SUM_{i=-n\atop i\neq
i_0}^n\tau(x_{i})\I[\psi,x_{i}]=o(\delta^{2s})\quad\text{as }\delta\rightarrow 0.\eeq
Similarly, from the behavior of $\psi$ at infinity we infer  that 
$$\delta^{2s}\psi(x_{i_0-1})\I[\tau,x_{i_0-1}],\,\delta^{2s}\psi(x_{i_0+1})\I[\tau,x_{i_0+1}]\quad\text{are }o(\delta^{2s})\quad\text{as }\delta\rightarrow 0.$$
This and \eqref{Itausum} imply that 
\beq\label{psixiItau} \delta^{2s}\SUM_{i=-n\atop i\neq
i_0}^n \psi(x_{i})\I[\tau,x_{i}]=o(\delta^{2s})\quad\text{as }\delta\rightarrow 0.\eeq

Let us now consider the term $\delta^{2s}\SUM_{i=-n\atop i\neq
i_0}^nB(\psi,\tau)(x_{i}).$ For $i\neq i_0-1,i_0,i_0+1$, using that $\tau(x_i)=0$, we have 
$$|B(\psi,\tau)(x_i)|\leq C(s)\int_\R\frac{|\psi(y)-\psi(x_i)|\tau(y)}{|x_i-y|^{1+2s}}dy\leq C\int_\R\frac{\tau(y)}{|x_i-y|^{1+2s}}dy=C\I[\tau,x_{i}].$$
Therefore, from \eqref{Itausum} we infer that 
\beq\label{Btau,psiineq} \delta^{2s}\SUM_{i=-n\atop i\neq
i_0,i_0\pm 1}^n|B(\psi,\tau)(x_{i})|\leq C\delta^{4s}.\eeq
Next, we remark that  for $\gamma\in\left(-\frac{1}{2},\frac{1}{2}\right]$ and $R=\frac{1}{2\delta|p_0|}$, either $x_{i_0-1}\in [-2R,-R]$ or $x_{i_0+1}\in(R,2R]$.
Suppose for instance that $x_{i_0-1}\in [-2R,-R]$ (i.e. $0\leq\gamma\leq\frac{1}{2}$). We have

\beqs\begin{split}B(\psi,\tau)(x_{i_0-1})&=C(s)\int_\R\frac{(\psi(y)-\psi(x_{i_0-1}))(\tau(y)-\tau(x_{i_0-1}))}{|x_{i_0-1}-y|^{1+2s}}dy\\&
=C(s)\int_{-2R}^{2R}\frac{(\psi(y)-\psi(x_{i_0-1}))\tau(y)}{|x_{i_0-1}-y|^{1+2s}}dy-C\tau(x_{i_0-1}))\I[\psi,x_{i_0-1}].
\end{split}\eeqs
We have already pointed out that $\I[\psi,x_{i_0-1}]=o(1)$ as $\delta\rightarrow0$. Let us consider the first term of the right-hand side of the last equality. 
Using that $R=\frac{1}{2\delta|p_0|}$, $x_{i_0-1}=\frac{-1+\gamma}{2\delta|p_0|}\in [-2R,-R]$ and estimate \eqref{psi'infinity}, we get 
\beqs\begin{split} \left|\int_{-2R}^{2R}\frac{(\psi(y)-\psi(x_{i_0-1}))\tau(y)}{|x_{i_0-1}-y|^{1+2s}}dy\right|&\leq 
\max_{[-2R,-R/2]}\psi'\int_{-2R}^{-\frac{R}{2}}\frac{1}{|x_{i_0-1}-y|^{2s}}dy\\&+C\int_{-\frac{R}{2}}^{2R}\frac{1}{|x_{i_0-1}-y|^{1+2s}}dy\\&
=C\max_{[-2R,-R/2]}\psi'\left[(x_{i_0-1}+2R)^{1-2s}+\left(-\frac{R}{2}-x_{i_0-1}\right)^{1-2s}\right]\\&
+C\left[\frac{1}{(2R-x_{i_0-1})^{2s}}-\frac{1}{(-\frac{R}{2}-x_{i_0-1})^{2s}}\right]\\&
\leq C\delta^{2s}.
\end{split}\eeqs
 We conclude that
 \beq\label{Btau,psii0-1}\delta^{2s}B(\psi,\tau)(x_{i_0-1})=o(\delta^{2s})\quad\text{as }\delta\rightarrow 0.\eeq
 Similarly we can prove that 
  \beq\label{Btau,psii0+1}\delta^{2s}B(\psi,\tau)(x_{i_0+1})=o(\delta^{2s})\quad\text{as }\delta\rightarrow 0.\eeq
  Estimates \eqref{Btau,psiineq}, \eqref{Btau,psii0-1} and \eqref{Btau,psii0+1} imply
  \beq\label{Btau,psi}\delta^{2s}\SUM_{i=-n\atop i\neq
i_0}^nB(\psi,\tau)(x_{i})=o(\delta^{2s})\quad\text{as }\delta\rightarrow 0.\eeq
In conclusion, putting together \eqref{ipsitausplitproduct}, \eqref{tauxiIpsi}, \eqref{psixiItau}  and \eqref{Btau,psi} we get
\beq\label{Ipsitauxineqest} \delta^{2s}\SUM_{i=-n\atop i\neq
i_0}^n\I[\psi\tau,x_{i}]=o(\delta^{2s})\quad\text{as }\delta\rightarrow 0.\eeq\\

\noindent {\bf Step 5': Bound on $W''(\phi(x_{i_0}))\delta^{2s}\SUM_{i=-n\atop
i\neq i_0}^n (\psi\tau)(x_i)$}\\

As in Step 2', using that  $\tau(x_i)=0$ for $i\neq i_0-1,i_0,i_0+1$ and that $\LIM_{x\rightarrow \pm \infty}\psi(x)=0$, we get
\beq\label{psitauxineq}\delta^{2s}\SUM_{i=-n\atop
i\neq i_0}^n (\psi\tau)(x_i)=
\delta^{2s}(\psi\tau)(x_{i_0-1})+\delta^{2s}(\psi\tau)(x_{i_0+1})=o(\delta^{2s})\quad\text{as }\delta\rightarrow 0.\eeq\\

\noindent {\bf Step  6': Bound on $\delta^{2s}B(\psi,\tau)(x_{i_0})-\delta^{2s} \psi(x_{i_0})\I[\tau,x_{i_0}]$}
Remember that $$x_{i_0}=\frac{\gamma}{\delta|p_0|},\quad|\gamma|\leq\frac{1}{2}.$$
Let us first assume $|\gamma|\leq\frac{1}{4},$ then $$|x_{i_0}|\leq \frac{1}{4\delta|p_0|}=\frac{R}{2},$$ and
\beqs\begin{split}\left| \I[\tau,x_{i_0}]\right|&=C\left| \int_\R\frac{\tau(y)-1}{|y-x_{i_0}|^{1+2s}}dy\right|\\&
= C\left|\int_{|y|> R}\frac{\tau(y)-1}{|y-x_{i_0}|^{1+2s}}dy\right|\\&
\leq C\int_{|y|> R}\frac{1}{|y-x_{i_0}|^{1+2s}}dy\\&
=\frac{C}{(x_{i_0}+R)^{2s}}+\frac{C}{(R-x_{i_0})^{2s}}\\&
\leq \frac{C}{R^{2s}}\\&=C\delta^{2s}.
\end{split}\eeqs
Then 
\beqs \delta^{2s} |\psi(x_{i_0})\I[\tau,x_{i_0}]|\leq C\delta^{4s}.\eeqs
Now let us assume $|\gamma|>\frac{1}{4}$. In this case $ \psi(x_{i_0})=o(1)$ as $\delta\rightarrow 0$, with $o(1)$ independent of $\gamma$ and then $ \delta^{2s} \psi(x_{i_0})\I[\tau,x_{i_0}]=o(\delta^{2s})$
 as $\delta\rightarrow 0$. We conclude that for any $\gamma\in\left(-\frac{1}{2},\frac{1}{2}\right]$ we have
\beq\label{psiitaixi0} \delta^{2s} \psi(x_{i_0})\I[\tau,x_{i_0}]=o(\delta^{2s})\quad\text{as }\delta\rightarrow 0.\eeq

Finally, let us consider the term $\delta^{2s}B(\psi,\tau)(x_{i_0})$. Again, if  $|\gamma|\leq\frac{1}{4},$ then 
\beqs\begin{split} \delta^{2s}|B(\psi,\tau)(x_{i_0})|&=\delta^{2s}C(s)\left|\int_\R\frac{(\psi(y)-\psi(x_{i_0}))(\tau(y)-1)}{|y-x_{i_0-1}|^{1+2s}}dy\right|\\&
\leq \delta^{2s}C\int_{|y|> R}\frac{1}{|y-x_{i_0}|^{1+2s}}dy\\&\leq C\delta^{4s}.\end{split}\eeqs
If $|\gamma|>\frac{1}{4},$ then either $x_{i_0}\in\left[-R,-\frac{R}{2}\right]$ or  $x_{i_0}\in\left[\frac{R}{2},R\right]$. 
Suppose for instance $x_{i_0}\in\left[-R,-\frac{R}{2}\right]$, then computations similar to those done in Step 5' for $B(\psi,\tau)(x_{i_0-1})$, show that $B(\psi,\tau)(x_{i_0})=o(1)$ as $\delta\rightarrow 0$.
We conclude that for any $\gamma\in\left(-\frac{1}{2},\frac{1}{2}\right]$ we have
\beq\label{Bpsitaixi0} \delta^{2s}B(\psi,\tau)(x_{i_0})=o(\delta^{2s})\quad\text{as }\delta\rightarrow 0.\eeq

\noindent {\bf Step  6'': Bound on the error $E$}\\

From \eqref{phisums<1/2estim} with $k=1$, and the fact that $\tau(x_i)=0$ for $i\neq i_0-1,i_0,i_0+1$ 
 it follows
that
\begin{equation}\label{NL5s<1/2}\left|\LIM_{n\rightarrow+\infty}E_2\right|=\left|\LIM_{n\rightarrow+\infty}O\left(\frac{L\delta^{2s}}{\al} +\delta^{2s}
\psi(x_{i_0})+\SUM_{i=-n\atop i\neq
i_0}^n\left[\widetilde{\phi}(x_i)+\delta^{2s}
\psi(x_i)\tau(x_i)\right]\right)^2\right|\leq C\delta^{4s}.
\end{equation}
Next, let us  estimate $E_1$. Remember that for $s<\frac{1}{2}$ we assume $W$ even, this  implies $W^{2k-1}(0)=0$ for any integer $k\geq 1$. Therefore
\beqs\begin{split} -E_1&=\SUM_{i=-n\atop i\neq i_0}^nW'(\widetilde{\phi}(x_i))-W''(0)\widetilde{\phi}(x_i)\\&
=W^{IV}(0)\SUM_{i=-n\atop i\neq i_0}^n(\widetilde{\phi}(x_i))^3+W^{VI}(0)\SUM_{i=-n\atop i\neq i_0}^n(\widetilde{\phi}(x_i))^5+...+
W^{2k_0}(0)\SUM_{i=-n\atop i\neq i_0}^n(\widetilde{\phi}(x_i))^{2k_0-1}\\&+\SUM_{i=-n\atop i\neq i_0}^nO((\widetilde{\phi}(x_i))^{2k_0+1}).\end{split}\eeqs
Fix $k_0$ such that $2s(2k_0+1)>1$, then by \eqref{phiinfinitys<1/2} the sequence $\SUM_{i=-n\atop i\neq i_0}^nO((\widetilde{\phi}(x_i))^{2k_0+1})$ 
is convergent since behaves like the series $\SUM_{i=1}^\infty\frac{1}{i^{2s(2k_0+1)}}$ and 
$$\SUM_{i=-n\atop i\neq i_0}^n|O((\widetilde{\phi})^{2k_0+1})|\leq C\delta^{2s(2k_0+1)}.$$
This estimate, together with \eqref{phisums<1/2estim} imply that 
\beq\label{E_1ests<1/2}|E_1|\leq C\delta^{4s}.\eeq

\noindent {\bf Step 7': Conclusion}\\
Therefore, from \eqref{NL1s<1/2}, \eqref{NL1psis<1/2}, \eqref{NL2s<1/2}, \eqref{Ipsitauxineqest}, \eqref{psitauxineq}, \eqref{psiitaixi0}, \eqref{Bpsitaixi0}, \eqref{NL5s<1/2} and 
\eqref{E_1ests<1/2} we conclude that
$$ \LIM_{n\rightarrow+\infty}NL_L^{\overline{\lam}_\delta^L}[s_{\delta,n}^L]=o(\delta^{2s})\quad\text{as }\delta\rightarrow0$$ 
where $\LIM_{\delta\rightarrow 0}\frac{o(\delta^{2s})}{\delta^{2s}}=0$, uniformly
 for $x\in\R$ and
Lemma \ref{lem::1011} for $s<\frac{1}{2}$ is
proved.\\

\subsection{Proof of Claims 1-6.} 
{\mbox{ }} \medskip

\noindent {\bf Proof of Claim 1.}\\
We have for $n>|i_0|$
\begin{equation*}\begin{split}\SUM_{i=-n\atop i\neq
i_0}^n\frac{x-i}{|x-i|^{1+2s}}&=\SUM_{i=-n}^{i_0-1}\frac{i_0+\gamma-i}{(i_0+\gamma-i)^{1+2s}}+\SUM_{i=i_0+1}^n\frac{i_0+\gamma-i}{(i-i_0-\gamma)^{1+2s}}\\&
=\SUM_{i=1}^{n+i_0}\frac{1}{(i+\gamma)^{2s}}-\SUM_{i=1}^{
n-i_0}\frac{1}{(i-\gamma)^{2s}}\\&
\end{split}\end{equation*}
Using that, for some $\theta_i\in(-1,1)$

$$\frac{(i-\gamma)^{2s}-(i+\gamma)^{2s}}{(i+\gamma)^{2s}(i-\gamma)^{2s}}=\frac{4s\gamma(i-\theta_i\gamma)^{2s-1}}{(i+\gamma)^{2s}(i-\gamma)^{2s}},$$ we get
\begin{equation*}\begin{split}
\SUM_{i=-n\atop i\neq
i_0}^n\frac{x-i}{|x-i|^{1+2s}}&=
\begin{cases}     \SUM_{i=1}^{
n}\frac{4s\gamma(i-\theta_i\gamma)^{2s-1}}{(i+\gamma)^{2s}(i-\gamma)^{2s}}, & \hbox{if } i_0=0\\
\noalign{\vskip6pt}
    \SUM_{i=1}^{
n-i_0}\frac{4s\gamma(i-\theta_i\gamma)^{2s-1}}{(i+\gamma)^{2s}(i-\gamma)^{2s}}+\SUM_{i=n-i_0+1}^{
n+i_0}\frac{1}{(i+\gamma)^{2s}}, & \hbox{if }i_0>0 \\
\noalign{\vskip6pt}
\SUM_{i=1}^{
n+i_0}\frac{4s\gamma(i-\theta_i\gamma)^{2s-1}}{(i+\gamma)^{2s}(i-\gamma)^{2s}}-\SUM_{i=n+i_0+1}^{
n-i_0}\frac{1}{(i-\gamma)^{2s}}, & \hbox{if }i_0<0 \\
\end{cases}\\ \,\\&
\rightarrow -\SUM_{i=1}^{
+\infty}\frac{4s\gamma(i-\theta_i\gamma)^{2s-1}}{(i+\gamma)^{2s}(i-\gamma)^{2s}}\quad\text{as }n\rightarrow+\infty.
\end{split}\end{equation*}

Let us prove the second limit of the
claim.
\begin{equation*}
\SUM_{i=-n}^{i_0-1}\frac{1}{|x-i|^{1+2s}}=
\SUM_{i=1}^{n+i_0}\frac{1}{(i+\gamma)^{1+2s}}\rightarrow
\SUM_{i=1}^{+\infty}\frac{1}{(i+\gamma)^{1+2s}}\quad\text{as
}n\rightarrow+\infty.\end{equation*} Finally
\begin{equation*}\SUM_{i=i_0+1}^{n}\frac{1}{|x-i|^{1+2s}}=
\SUM_{i=1}^{n-i_0}\frac{1}{(i-\gamma)^{1+2s}}\rightarrow
\SUM_{i=1}^{+\infty}\frac{1}{(i-\gamma)^{1+2s}}\quad\text{as
}n\rightarrow+\infty,
\end{equation*}
and the claim  is proved.
\medskip

\noindent {\bf Proof of Claim 2.}\\

When $s<\frac{1}{2}$, we assume that $W$ is even and this implies  that the function
$$\phi(x)-\frac{1}{2}$$ is odd, which means that $\phi$ satisfies
$$\phi(-x)=-\phi(x)+1,$$ and therefore for any integer $k\geq 1$
$$[\phi(-x)]^{2k-1}=[-\phi(x)+1]^{2k-1}=-[\phi(x)-1]^{2k-1}.$$
For simplicity, let us assume $i_0>0$. We have 
\beqs\begin{split}
\SUM_{i=-n\atop
i\neq i_0}^n[\widetilde{\phi}(x_i)]^{2k-1}&=\SUM_{i=-n}^{i_0-1}[\phi(x_i)-1]^{2k-1}+\SUM_{i=i_0+1}^n[\phi(x_i)]^{2k-1}\\&
=\SUM_{i=1}^{n+i_0}\left[\phi\left(\frac{i+\gamma}{\delta|p_0|}\right)-1\right]^{2k-1}+\SUM_{i=1}^{n-i_0}\left[\phi\left(-\frac{i-\gamma}{\delta|p_0|}\right)\right]^{2k-1}\\&
=\SUM_{i=1}^{n+i_0}\left[\phi\left(\frac{i+\gamma}{\delta|p_0|}\right)-1\right]^{2k-1}-\SUM_{i=1}^{n-i_0}\left[\phi\left(\frac{i-\gamma}{\delta|p_0|}\right)-1\right]^{2k-1}\\&
=\SUM_{i=1}^{n-i_0}\left\{\left[\phi\left(\frac{i+\gamma}{\delta|p_0|}\right)-1\right]^{2k-1}-\left[\phi\left(\frac{i-\gamma}{\delta|p_0|}\right)-1\right]^{2k-1}\right\}\\&
+\SUM_{i=n-i_0+1}^{n+i_0}\left[\phi\left(\frac{i+\gamma}{\delta|p_0|}\right)-1\right]^{2k-1}\\&
=\SUM_{i=1}^{n-i_0}\left\{\left[\phi\left(\frac{i+\gamma}{\delta|p_0|}\right)-\phi\left(\frac{i-\gamma}{\delta|p_0|}\right)\right]\right.\\&\cdot
\left.\SUM_{l=0}^{2k-2}\left(\phi\left(\frac{i+\gamma}{\delta|p_0|}\right)-1\right)^{l}\left(\phi\left(\frac{i-\gamma}{\delta|p_0|}\right)-1\right)^{2k-2-l}\right\}\\&
+\SUM_{i=n-i_0+1}^{n+i_0}\left[\phi\left(\frac{i+\gamma}{\delta|p_0|}\right)-1\right]^{2k-1}\\&
=\SUM_{i=1}^{n-i_0}\phi'\left(\frac{i+\theta_i\gamma}{\delta|p_0|}\right)\frac{2\gamma}{\delta|p_0|}\SUM_{l=0}^{2k-2}\left(\phi\left(\frac{i+\gamma}{\delta|p_0|}\right)-1\right)^{l}\left(\phi\left(\frac{i-\gamma}{\delta|p_0|}\right)-1\right)^{2k-2-l}\\&
+\SUM_{i=n-i_0+1}^{n+i_0}\left[\phi\left(\frac{i+\gamma}{\delta|p_0|}\right)-1\right]^{2k-1}\\&
\end{split}\eeqs for some $\theta_i\in(-1,1)$. 
Therefore, using \eqref{phi'infinity} and  \eqref{phiinfinitys<1/2}, we get
\beqs\begin{split} \left|\SUM_{i=-n\atop
i\neq i_0}^n[\widetilde{\phi}(x_i)]^{2k-1}\right|
&\leq C\delta^{2s(2k-1)}|\gamma| \SUM_{i=1}^{n-i_0}\frac{1}{(i-|\gamma|)^{1+2s}}\SUM_{l=0}^{2k-2}\frac{1}{(i-|\gamma|)^{2s(2k-2)}}
\\&+C\SUM_{i=n-i_0+1}^{n+i_0}\frac{1}{|i+\gamma|^{2s(2k-1)}}\\&
\leq Ck\delta^{2s(2k-1)}|\gamma| \SUM_{i=1}^{n-i_0}\frac{1}{(i-|\gamma|)^{1+2s(2k-1)}}+C\SUM_{i=n-i_0+1}^{n+i_0}\frac{1}{|i+\gamma|^{2s(2k-1)}}.
\end{split}\eeqs
Passing to the limit as $n\rightarrow+\infty$, we get \eqref{phisums<1/2estim}

Next, let us turn to the proof of \eqref{Itausum}.  For $i\neq i_0-1,i_0,i_0+1$, and $R=\frac{1}{2\delta|p_0|}$
$$|x_i|=\frac{|i_0+\gamma-i|}{\delta|p_0|}\geq \frac{3}{2\delta|p_0|}>2R,$$ therefore $\tau(x_i)=0$ and
\beqs\begin{split}0\leq\I[\tau,x_i]&=\int_\R\frac{\tau(y)}{|x_i-y|^{1+2s}}dy\\&
=\int_{-2R}^{2R}\frac{\tau(y)}{|x_i-y|^{1+2s}}dy\\&
\leq \int_{-2R}^{2R}\frac{1}{|x_i-y|^{1+2s}}dy\\&
=\int_{|x_i|-2R}^{|x_i|+2R}\frac{1}{y^{1+2s}}dy\\&
=2s\left[\frac{1}{(|x_i|-2R)^{2s}}-\frac{1}{(|x_i|+2R)^{2s}}\right]\\&
=16s^2R\frac{(|x_i|+2R\theta_i)^{2s-1}}{(|x_i|-2R)^{2s}(|x_i|+2R)^{2s}},
\end{split}\eeqs
for some $\theta_i\in(-1,1)$. Therefore, for $R=\frac{1}{2\delta|p_0|}$ we have
\beqs\begin{split}
0\leq\SUM_{i=-n\atop
i\neq i_0,i_0\pm1}^n\I[\tau,x_i]&\leq 8s^2\delta^{2s}|p_0|^{2s}\SUM_{i=-n}^{i_0-2}\frac{(i_0+\gamma-i+\theta_i)^{2s-1}}{(i_0+\gamma-i-1)^{2s}(i_0+\gamma-i+1)^{2s}}\\&
+8s^2\delta^{2s}|p_0|^{2s}\SUM_{i=i_0+2}^{n}\frac{(-i_0-\gamma+i+\theta_i)^{2s-1}}{(-i_0-\gamma+i-1)^{2s}(-i_0-\gamma+i+1)^{2s}}\\&
=C\delta^{2s}\SUM_{i=2}^{n+i_0}\frac{(i+\gamma+\theta_i)^{2s-1}}{(i+\gamma-1)^{2s}(i+\gamma+1)^{2s}}\\&
+C\delta^{2s}\SUM_{i=2}^{n-i_0}\frac{(i-\gamma+\theta_i)^{2s-1}}{(i-\gamma-1)^{2s}(i-\gamma+1)^{2s}}\\&
\leq C\delta^{2s}\SUM_{i=2}^{n+|i_0|}\frac{1}{\left(i+\frac{1}{2}\right)^{2s}\left(i-\frac{3}{2}\right)},
\end{split}\eeqs
which implies \eqref{phisums<1/2estim}.
\medskip

\noindent {\bf Proof of Claim 3.}\\
Fix
$x\in\R$ and let $i_0\in\Z$ be the closest integer to $x$ such
that $x=i_0+\gamma$, with
$\gamma\in\left(-\frac{1}{2},\frac{1}{2}\right]$ and $|x-i|\geq
\frac{1}{2}$ for $i\neq i_0$. Let $\delta$ be so small that
$\frac{1}{\delta|p_0|}\geq 2$, then $\frac{|x-i|}{\delta|p_0|}\geq
1$ for $i\neq i_0$. Let us first assume $s\geq \frac{1}{2}$. Then, for $n>|i_0|$ using  \eqref{phiinfinity}
and
 \eqref{psiinfinity} we get
\begin{equation*}\begin{split}s_{\delta,n}^L(x)&= \phi(x_{i_0})+\delta^{2s}\psi(x_{i_0})+
i_0+\SUM_{i=-n}^{i_0-1}\left[\phi(x_i)-1+\delta^{2s}\psi(x_i)\right]
\\&+\SUM_{i=i_0+1}^{n}\left[\phi(x_i)+\delta^{2s}\psi(x_i)\right] \\&
\leq C+i_0
-\left(\frac{1}{2s\al}-\delta^{2s}
K_2\right)\delta^{2s}|p_0|^{2s}\SUM_{i=-n\atop i\neq i_0}^{n}\frac{x-i}{|x-i|^{1+2s}}\\&+(K_1+\delta^{2s}
K_3)\delta^{1+2s}|p_0|^{1+2s}\SUM_{i=-n\atop i\neq i_0}^{n}\frac{1}{|x-i|^{1+2s}},
\end{split}\end{equation*}
and
\begin{equation*}\begin{split}
s_{\delta,n}^L(x)&\geq
C+i_0
-\left(\frac{1}{2s\al}-\delta^{2s}
K_2\right)\delta^{2s}|p_0|^{2s}\SUM_{i=-n\atop i\neq i_0}^{n}\frac{x-i}{|x-i|^{1+2s}}\\&-(K_1+\delta^{2s}
K_3)\delta^{1+2s}|p_0|^{1+2s}\SUM_{i=-n\atop i\neq i_0}^{n}\frac{1}{|x-i|^{1+2s}}.\end{split}\end{equation*}
Then from Claim 1 we
conclude that the sequence $\{s_{\delta,n}^L(x)\}_n$ is convergent as $n\rightarrow+\infty$, moreover for $x=i_0+\gamma$, we have

\beq\label{sdeltan-xbddclaim3}|s_{\delta,n}^L(x)-x|\leq C.
\eeq

When $s<\frac{1}{2}$,  the convergence of $\SUM_{i=-n}^{n}\phi(x_i)-n$ follows from \eqref{phisums<1/2estim} for $k=1$. The sum
$\SUM_{i=-n}^{n}\psi(x_i)\tau(x_i)$ is actually the  sum of only three terms, since as we have seen in the proof of Claim 2,   $\tau(x_i)=0$ for $i\neq i_0-1,i_0,i_0+1$. 
This concludes the proof of Claim 3.
\medskip

\noindent {\bf Proof of Claim 4.}\\
To prove the
uniform convergence, it suffices to show that
$\{(s_{\delta,n}^L)'(x)\}_n$ is a Cauchy sequence uniformly on
compact sets. Let us consider a bounded interval $[a,b]$ and let
$x\in[a,b]$. Let us first assume $s\geq\frac{1}{2}$. For $\frac{1}{\delta|p_0|}\geq 2$ and $k>
m>1/2+\max\{|a|,|b|\}$, by \eqref{phi'infinity} and
\eqref{psi'infinity} we have
\begin{equation*}\begin{split}(s_{\delta,k}^L)'(x)-(s_{\delta,m}^L)'(x)&=
\frac{1}{\delta|p_0|}\SUM_{i=-k}^{-m-1}\left[\phi'(x_i)+\delta^{2s}\psi'(x_i)\right]
+\frac{1}{\delta|p_0|}\SUM_{i=m+1}^{k}\left[\phi'(x_i)+\delta^{2s}\psi'(x_i)\right]\\&
\leq(K_1+\delta^{2s}
K_3)\delta^{2s}|p_0|^{2s}\left[\SUM_{i=-k}^{-m-1}\frac{1}{|x-i|^{1+2s}}+\SUM_{i=m+1}^{k}\frac{1}{|x-i|^{1+2s}}\right]\\&
\leq(K_1+\delta^{2s}
K_3)\delta^{2s}|p_0|^{2s}\left[\SUM_{i=-k}^{-m-1}\frac{1}{|a-i|^{1+2s}}+\SUM_{i=m+1}^{k}\frac{1}{|b-i|^{1+2s}}\right],
\end{split}\end{equation*}and
\begin{equation*}(s_{\delta,k}^L)'(x)-(s_{\delta,m}^L)'(x)\geq
-K_3\delta^{4s}|p_0|^{2s}\left[\SUM_{i=-k}^{-m-1}\frac{1}{|a-i|^{1+2s}}+\SUM_{i=m+1}^{k}\frac{1}{|b-i|^{1+2s}}\right].\end{equation*}
Then by Claim 1
$$\sup_{x\in[a,b]}|(s_{\delta,k}^L)'(x)-(s_{\delta,m}^L)'(x)|\rightarrow0
\quad\text{as }k,m\rightarrow+\infty.$$

When $s<\frac{1}{2}$, the convergence of $\SUM_{i=-n\atop i\neq i_0}^{n}\phi'(x_i)$ is again consequence of estimate  \eqref{phi'infinity},
and the convergence of $\SUM_{i=-n\atop i\neq i_0}^{n}(\psi\tau)'(x_i)$ comes from the fact that this is actually the sum of three terms, being $\tau(x_i)=\tau'(x_i)=0$ 
for $i\neq i_0-1,i_0,i_0+1$. Claim 4 is therefore proved. 
\medskip

\noindent {\bf Proof of Claim 5.}\\
Claim 5 can be proved like Claim 4, using \eqref{phi''infinity}, 
\eqref{psi''infinity} and the properties of $\tau$.\medskip

\noindent {\bf Proof of Claim 6.}\\
Let us first assume $s\geq\frac{1}{2}$. We have
$$\I[\phi]=W'(\phi)=W'(\widetilde{\phi})=W''(0)\widetilde{\phi}+O(\widetilde{\phi})^2.$$
We note that, since $s\geq\frac{1}{2}$,
if $|x|\geq 1$ then $$(\phi(x))^2\leq\frac{C}{|x|^{4s}}\leq \frac{C}{|x|^{1+2s}}.$$
Let $x=i_0+\gamma$ with
$\gamma\in\left(-\frac{1}{2},\frac{1}{2}\right]$, and $n>|i_0|$. From \eqref{phiinfinity} we get
\begin{equation*}\begin{split}\SUM_{i=-n}^n\I[\phi,x_i]&=I[\phi,x_{i_0}]+\SUM_{i=-n\atop i\neq i_0}^n\I[\phi,x_i]\\&
=I[\phi,x_{i_0}]+\SUM_{i=-n\atop i\neq i_0}^n[\al\widetilde{\phi}(x_i)+O(\widetilde{\phi}(x_i))^2]\\&
\leq I[\phi,x_{i_0}]
-\frac{\delta^{2s}|p_0|^{2s}}{2s}\SUM_{i=-n\atop i\neq i_0}^n\frac{x-i}{|x-i|^{1+2s}}+C\SUM_{i=-n\atop i\neq i_0}^n\frac{1}{|x-i|^{1+2s}},
\end{split}\end{equation*} for some $C>0$ and
\begin{equation*}\begin{split}\SUM_{i=-n}^n\I[\phi,x_i]\geq I[\phi,x_{i_0}]
-\frac{\delta^{2s}|p_0|^{2s}}{2s}\SUM_{i=-n\atop i\neq i_0}^n\frac{x-i}{|x-i|^{1+2s}}-C\SUM_{i=-n\atop i\neq i_0}^n\frac{1}{|x-i|^{1+2s}}.
 \end{split}\end{equation*} Then, by Claim 1
$\SUM_{i=-n}^n\I[\phi,x_i]$ converges as $n\rightarrow+\infty$. 

Let us consider now  $\SUM_{i=-n}^n\I[\psi,x_i]$. From the following estimate
\begin{equation*}\begin{split}\I[\psi]&=W''(\widetilde{\phi})\psi+\frac{L}{\al}(W''(\widetilde{\phi})-W''(0))+c\phi'
\\&=W''(0)\psi+\frac{L}{\al}W'''(0)\widetilde{\phi}+O(\widetilde{\phi})\psi+O(\widetilde{\phi})^2+c\phi',\end{split}\end{equation*}
\eqref{phiinfinity}, \eqref{phi'infinity} and
\eqref{psiinfinity} we get
\begin{equation*}\begin{split}
\SUM_{i=-n}^n\I[\psi,x_i]\leq I[\psi,x_{i_0}]
+\widetilde{C}\SUM_{i=-n\atop i\neq i_0}^n\frac{x-i}{|x-i|^{1+2s}}+C\SUM_{i=-n\atop i\neq i_0}^n\frac{1}{|x-i|^{1+2s}},
\end{split}\end{equation*}and

\begin{equation*}\begin{split}
\SUM_{i=-n}^n\I[\psi,x_i]\geq I[\psi,x_{i_0}]
+\widetilde{C}\SUM_{i=-n\atop i\neq i_0}^n\frac{x-i}{|x-i|^{1+2s}}-C\SUM_{i=-n\atop i\neq i_0}^n\frac{1}{|x-i|^{1+2s}},
\end{split}\end{equation*}
 for some $\widetilde{C}\in\R$ and $C>0$, which ensures the convergence of
$\SUM_{i=-n}^n\I[\psi,x_i]$.

Now, let us assume $s<\frac{1}{2}$.  Fix $k_0$ such that $2s(2k_0+1)>1$. Since $W$ is even,  $W^{2k+1}(0)=0$ for any integer $k\geq1$. Then
\beqs\I[\phi]=W'(\widetilde{\phi})=W''(0)\widetilde{\phi}+W^{IV}(0)(\widetilde{\phi})^3+...+W^{2k_0}(0)(\widetilde{\phi})^{2k_0-1}+O((\widetilde{\phi})^{2k_0+1}).\eeqs
Therefore, for $x=i_0+\gamma$
\beqs\begin{split} \SUM_{i=-n}^n\I[\phi,x_i]&=\I[\phi,x_{i_0}]+\SUM_{i=-n\atop i\neq i_0}^n \I[\phi,x_i]\\&
=\I[\phi,x_{i_0}]+W''(0)\SUM_{i=-n\atop i\neq i_0}^n\widetilde{\phi}(x_i)+W^{IV}(0)\SUM_{i=-n\atop i\neq i_0}^n(\widetilde{\phi}(x_i))^3+...\\&
+W^{2k_0}(0)\SUM_{i=-n\atop i\neq i_0}^n(\widetilde{\phi}(x_i))^{2k_0-1}+\SUM_{i=-n\atop i\neq i_0}^nO((\widetilde{\phi}(x_i))^{2k_0+1}).
\end{split}\end{equation*}
The sequence $\SUM_{i=-n\atop i\neq i_0}^nO((\widetilde{\phi}(x_i))^{2k_0+1})$ is convergent since, by \eqref{phiinfinitys<1/2} behaves like $\SUM_{i=1}^n\frac{1}{i^{2s(2k_0+1)}}$ which is convergent being the exponent $2s(2k_0+1)$ greater than 1. The convergence of the remaining sequences is assured by \eqref{phisums<1/2estim}. 

Finally, let us consider $\SUM_{i=-n}^n\I[\psi\tau,x_i]$. The following formula, which can be found for instance in \cite{bpsv} page 7, holds true
\beqs \I[\psi\tau,x_i]=\tau(x_i) \I[\psi,x_i]+\psi(x_i) \I[\tau,x_i]-B(\psi,\tau)(x_i),\eeqs
where
$$B(\psi,\tau)(x_i)=C(s)\int_\R\frac{(\psi(y)-\psi(x_i))(\tau(y)-\tau(x_i))}{|x-y|^{1+2s}}dy.$$
We remark that 
\beqs|B(\psi,\tau)(x_i)|\leq C\int_\R\frac{|\tau(y)-\tau(x_i)|}{|x-y|^{1+2s}}dy=C\int_\R\frac{\tau(y)}{|x-y|^{1+2s}}dy=C\I[\tau,x_i],\eeqs for $i\neq i_0-1,i_0,i_0+1$. Indeed, as we have already pointed out $\tau(x_i)=0$ for these indices. Therefore the sequences $\SUM_{i=-n}^n\psi(x_i) \I[\tau,x_i]$ and $\SUM_{i=-n}^nB(\psi,\tau)(x_i)$ converge by \eqref{Itausum}. 
Also  $\SUM_{i=-n}^n\tau(x_i) \I[\psi,x_i]$ is the sum of only three terms and then we can conclude 
that $\SUM_{i=-n}^n\I[\psi\tau,x_i]$ is convergent as $n\rightarrow+\infty$. This concludes the proof of Claim 6.

\end{document}